\documentclass{article}

\usepackage{comment}
\usepackage{mathtools}%
\usepackage{amsfonts}%
\usepackage{amssymb}%
\usepackage{amsthm}

\usepackage[notcite,notref, final]{showkeys}

\usepackage{graphicx}
\usepackage{xspace}
\usepackage{url}
\usepackage[english]{babel}

\usepackage[hidelinks]{hyperref}
\usepackage{cleveref}
\usepackage{float}
\usepackage[all]{xy}
\usepackage{etoolbox}
\usepackage{latexsym}
\usepackage{enumerate}
\usepackage{cite}

\usepackage{color}
\usepackage{amsmath}
\usepackage{amssymb}

\theoremstyle{plain}
\newtheorem{theorem}{Theorem}[section]
\newtheorem{lemma}[theorem]{Lemma}
\newtheorem{proposition}[theorem]{Proposition}

\newtheorem{corollary}[theorem]{Corollary}

\theoremstyle{definition}
\newtheorem{definition}[theorem]{Definition}
\newtheorem{remark}[theorem]{Remark}
\newtheorem{example}[theorem]{Example}

\makeatletter
\newcommand*\NoIndentAfterEnv[1]{\AfterEndEnvironment{#1}{\par\@afterindentfalse\@afterheading}}
\makeatother
\NoIndentAfterEnv{remark}
\NoIndentAfterEnv{theorem}
\NoIndentAfterEnv{definition}
\NoIndentAfterEnv{corollary}
\NoIndentAfterEnv{example}
\NoIndentAfterEnv{lemma}
\NoIndentAfterEnv{proposition}
\NoIndentAfterEnv{proof}

\newcommand{\ext}{\operatorname{ext}}

\newcommand{\field}[1]{\mathbb{#1}}

\newcommand{\N}{\field{N}}

\newcommand{\R}{\field{R}}
\newcommand{\C}{\field{C}}
\renewcommand{\H}{\field{H}}
\newcommand{\D}{\field{D}}
\renewcommand{\S}{\mathbb{S}}

\renewcommand{\Re}{\mathrm{Re}}

\newcommand{\lhol}{\mathcal{SH}_L}

\newcommand{\intrin}{\mathcal{N}}
\newcommand{\sderiv}[1][]{\partial_{S#1}}

\newcommand{\supp}{\operatorname{supp}}
\newcommand{\borel}{\mathsf{B}}
\newcommand{\hardy}{\mathcal{H}}
\newcommand{\BMOSH}{\mathrm{BMOSH}}
\newcommand{\BMOA}{\mathrm{BMOA}}
\newcommand{\VMOSH}{\mathrm{VMOSH}}
\newcommand{\VMOA}{\mathrm{VMOA}}
\newcommand{\bloch}{\mathcal{B}}
\newcommand{\dirichlet}{\mathcal{D}}

\date{}
\title{\bf BMO- and VMO-spaces of slice hyperholomorphic functions}

\author{
Jonathan Gantner \\
Politecnico di Milano\\
Dipartimento di Matematica\\
Via E. Bonardi, 9\\
20133 Milano, Italy\\
jonathan.gantner@polimi.it\\
\and
J. Oscar Gonz\'{a}lez-Cervantes\\
Departamento de Matem\'{a}ticas\\
E.S.F.M. del I.P.N.\\
 07338 Mexico, D.F., Mexico\\
 jogc200678@gmail.com\and
Tim Janssens\\
University of Antwerp\\
Department of Mathematics\\
Middelheimlaan, 1\\
2020 Antwerp, Belgium\\
tim.janssens4@uantwerpen.be
}

\begin{document}

\maketitle
\begin{abstract}
In this paper we continue the study of important Banach spaces of slice hyperholomorphic functions on the quaternionic unit ball by investigating the BMO- and VMO-spaces of slice hyperholomorphic functions. We discuss in particular conformal invariance and a refined characterization of these spaces in terms of Carleson measures. Finally we show the relations with the Bloch and Dirichlet space and the duality relation with the Hardy space $\hardy^{1}(\D)$.

The importance of these spaces in the classical theory is well known. It is therefore worthwhile to study their slice hyperholomorphic counterparts, in particular because slice hyperholomorphic functions were found to have several applications  in operator theory and Schur analysis.
\end{abstract}
\section{Introduction}
The notion of holomorphicity can be generalized in several ways to functions that are defined on the skew-field of quaternions. The most studied notion of generalized holomorphicity in this setting is Fueter regularity, which was introduced by Fueter in 1932 in \cite{Fueter1,Fueter2} and which lead to a rich and nowadays well developed theory. In recent times however, the notion of slice hyperholomorphicity introduced in \cite{GentiliStruppa} in 2006 has attracted the interest of mathematicians. This was in particular due to its applications in the theory of quaternionic linear operators: the slice hyperholomorphic Cauchy kernel lead to the discovery of the proper notion of spectrum for quaternionic linear operators \cite{SHBook}. As a consequence, it was possible to define a generalized Riesz-Dunford functional calculus for such operators, the so-called $S$-functional calculus in \cite{SHBook}, which is based on the theory of slice hyperholomorphic functions. Furthermore it was possible to develop a continuous functional calculus for bounded normal operators on quaternionic Hilbert spaces in \cite{GhilPer} and to prove the spectral theorems for unitary and for unbounded normal quaternionic linear operators \cite{STUnit, STNormal}. These methods are fundamental in quaternionic quantum mechanics \cite{Adler}. Moreover the higher dimensional version of slice hyperholomorphicity on Clifford algebras introduced in \cite{CliffHolo} allowed to develop the $S$-functional calculus also for $n$-tuples of non commuting operators \cite{CliffCalc}.

Besides operator theory, slice hyperholomorphic functions also appear in Schur analysis. (For an overview, we refer to \cite{AlpayBook}.) The theory is in particular fundamental for the realization of slice hyperholomorphic Schur functions \cite{SHSchur1, SHSchur2, SHSchur3}.

The theory of slice hyperholomorphic functions itself is nowadays developing rapidly in several directions. For an overview we refer to \cite{SHBook} and \cite{GentiliBook} as well as to \cite{ACS} and \cite{CSS} and the references therein.

Several Banach spaces of holomorphic functions play an important role in complex operator theory \cite{kehe}. Because of the fundamental role that slice hyperholomorphic functions play in quaternionic operator theory, the slice hyperholomorphic counterparts of these spaces are objects worth to be studied in detail. The study of these spaces has begun recently. The slice hyperholomorphic  Hardy spaces $\hardy^2(\D)$ on the unit ball and $\hardy^2(\H^+)$ on the quaternionic half space $\H^+$, which consists of those quaternions that have positive real part, were introduced in \cite{SHSchur1, SHSchur2} and Hardy spaces $\hardy^p(\D)$ with $0<p\leq+\infty$ on the unit ball were studied in detail in \cite{HardyRef}. Slice hyperholomorphic Bergman spaces were treated in \cite{Bergman1,Bergman2,Bergman3,Bergman4} and slice hyperholomorphic Fock spaces in \cite{Fock}. Furthermore \cite{CGS3} contains a first study on slice hyperholomorphic Bloch, Besov and weighted Bergman spaces and on the Dirichlet space in this setting.
 
 The aim of this paper is to continue this work by a detailed study of the spaces of slice hyperholomorphic functions of bounded and vanishing mean oscillation. First steps in this direction have been done in \cite{Arcozzi}, where the slice hyperholomorphic BMO-space appears in connection with quaternionic Hankel operators.
 
 The plan of the paper is as follows: in \Cref{PrelSec} we give a brief introduction to the theory of slice hyperholomorphic functions and recall the results that we need later on. In \Cref{CarlMeasSec} we define Carleson measures, vanishing Carleson measures, and their slice decompositions and give a detailed discussion of the relations between these two concepts. Carleson measures and their slice decompositions also appeared in the context of Hardy spaces in \cite{sabsar} and were in this context originally introduced in \cite{Arcozzi}.

 In \Cref{BMOSect} we introduce the slice hyperholomorphic BMO-space and show several of its properties: we discuss the relation between slicewise and global BMO-norms and show that the space is invariant under $i$-composition with slice hyperholomorphic M\"obius transformations. Moreover we show that it is possible to define a norm that is equivalent to the BMO-norm and invariant under $i$-composition with slice hyperholomorphic M\"obius transformations generated by a fixed complex plane. Finally, we give a refined characterization of slice hyperholomorphic BMO-functions in terms of Carleson measures. 

\Cref{VMOSec} studies the slice hyperholomorphic VMO-space and gives several characterizations of functions in this space. In particular we show that it is---as in the complex case---the BMO-closure of the set of slice hyperholomorphic polynomials, which coincides also with the closure of the set of those slice hyperholomorphic functions that can be extended continuously to the boundary of the unit ball. Then we also characterize the functions in this space in terms of Carleson measures and we discuss power series with Hadamard gaps in this setting.

Finally, we show in \Cref{RelSec} that several relations between the BMO-space resp. the VMO-space and other function spaces also hold true in the slice hyperholomorphic setting. Precisely, we discuss the Bloch and the Drichlet space and show that the duality relations with the Hardy space $\hardy^{1}(\D)$ hold true in our setting.

\section{Preliminaries}\label{PrelSec}
The skew-field of quaternions $\H$ consists of the real vector space of elements of the form
$x=\xi_0 + \sum_{\ell=1}^3\xi_\ell e_\ell$ with $\xi_\ell\in\R $, which is endowed with a product such that $1$ is the multiplicative unit and the imaginary units $e_1$, $e_2$ and $e_3$ satisfy
$e_{\ell}^2 = -1$ and $ e_{\ell}e_{\kappa} =  - e_{\kappa}e_{\ell}$ for $\kappa,\ell \in\{1,2,3\}$ with $\kappa\neq \ell$.
The real part of a quaternion $x = \xi_0 + \sum_{\ell=1}^3\xi_\ell e_\ell$ is defined as $\Re(x) := \xi_0$,
its imaginary part as
 $\underline{x} := \sum_{\ell=1}^3\xi_\ell e_\ell$ and its conjugate as $\overline{x} := \Re(x) - \underline{x}$.
Each element of the set
\[\S := \{ x\in\H: \Re(x) = 0, |x| = 1 \}\]
 is a square-root of $-1$ and is therefore called an imaginary unit. For any $i\in\S$, the subspace
 $\C_i := \{x_0 + i x_1: x_1,x_2\in\R\}$ is isomorphic to the field of complex numbers. For $i,j\in\S$ with $i\perp j$, set $k=ij = -ji$. Then $1$, $i$, $j$ and $k$ form an
 orthonormal basis of $\H$ as a real vector space and $1$ and $j$ form an orthonormal basis of $\H$ as a left or right vector space over the complex plane $\C_i$, that is
 \[ \H = \C_i + \C_i j\quad\text{and}\quad \H = \C_i + j\C_i.\]
  Any quaternion $x$ belongs to such a complex plane: if we set
 \[i_x := \begin{cases}\underline{x}/|\underline{x}|,& \text{if  }\underline{x} \neq 0 \\ \text{any }i\in\S, \quad&\text{if }\underline{x}  = 0,\end{cases}\]
 then $x = x_0 + i_x x_1$ with $x_0 =\Re(x)$ and $x_1 = |\underline{x}|$. The set
 \[
 [x] := \{x_0 + ix_1: i\in\S\},
 \]
is a 2-sphere, that reduces to a single point if $x$ is real.
\begin{definition}
A set $U\subset\H$ is called
\begin{enumerate}[(i)]
\item axially symmetric if $[x]\subset U$ for any $x\in U$ and
\item a slice domain if $U$ is open, $U\cap\R\neq 0$ and $U\cap\C_i$ is a domain for any $i\in\S$.
\end{enumerate}
\end{definition}

\begin{definition}\label{LHolDef}
Let $U\subset\H$ be an axially symmetric open set. A real differentiable function $f: U\to\H$ is called (left) slice hyperholomorphic if it has the form
\begin{equation}\label{LSHol}
 f(x) = \alpha(x_0,x_1) + i_x\beta(x_0,x_1),\quad \forall x = x_0 + i_x x_1 \in U\end{equation}
and the functions $\alpha$ and $\beta$, which take values in $\H$, satisfy the compatibility condition
\begin{equation}\label{SymCond}
\alpha(x_0,-x_1) = \alpha(x_0,x_1)\qquad \beta(x_0,-x_1) = - \beta(x_0,x_1)
\end{equation}
and the Cauchy-Riemann-system
\begin{equation}\label{CR}\begin{split}
\frac{\partial}{\partial x_0} \alpha(x_0,x_1) &= \frac{\partial}{\partial x_1} \beta(x_0,x_1)\\
\frac{\partial}{\partial x_0}\beta(x_0,x_1) &= -\frac{\partial}{\partial x_1}\alpha(x_0,x_1).
\end{split}\end{equation}
If in addition the functions $\alpha$ and $\beta$ take values in $\R$, then $f$ is called intrinsic.

The sets of slice hyperholomorphic and intrinsic functions on $U$ are denoted by $\lhol(U)$ and $\intrin(U)$, respectively.
\end{definition}

\begin{remark}
In the literature slice hyperholomorphic functions on $\H$ are often also called slice regular. Moreover, it is also possible to define the notion of right slice hyperholomorphicity, in which case the functions have the form $ f(x) = \alpha(x_0,x_1) + \beta(x_0,x_1)i_x,$
such that the functions $\alpha$ and $\beta$ satisfy \eqref{SymCond} and \eqref{CR}. This leads to an equivalent theory.
\end{remark}
Intrinsic functions play an important role in the theory of slice hyperholomorphic functions because composition and multiplication with such functions preserve slice hyperholomorphicity, which is in general not true.
\begin{corollary}
Let $U$ be an axially symmetric open set.
\begin{enumerate}[(i)]
\item If $f\in\intrin(U)$ and $g\in\lhol(U)$, then $fg\in\lhol(U)$.
\item if $g\in\intrin(U)$ and $f\in\lhol(g(U))$, then $f\circ g \in \lhol(U)$. 
\end{enumerate}
\end{corollary}
An identity principle with a slice wise condition holds true, but only on axially symmetric slice domains.
\begin{theorem}[Identity Principle]\label{IDPrinciple}
Let $U$ be an axially symmetric slice domain, let $f\in\lhol(U)$ and let $\mathcal{Z}_f$ be the set of zeros of $f$. If there exists $i\in\S$ such that $\mathcal{Z}_{f,i} := \mathcal{Z}_{f}\cap\C_i$ has an accumulation point in $U_i:=U\cap\C_i$, then $f\equiv 0$.
\end{theorem}
The most important examples of slice hyperholomorphic functions are polynomials and power series with coefficients on the right, i.e. of the form $f(x) = \sum_{n=0}^{+\infty} x^na_n$ with $a_n\in\H$. A power series of this form is intrinsic if and only if the coefficients are real. 

Conversely, any slice hyperholomorphic function can be expanded into a power series at any real point.
\begin{definition}
Let $f\in\lhol(U)$. The slice derivative of $f$ is the function defined by
\begin{equation}\label{sderiv}
 \sderiv f(s) := \lim_{\C_{i_s}\ni p\to s} (p-s)^{-1}(f(p)-f(s)),
\end{equation}
where $\lim_{\C_{i_s}\ni p \to  s}g(p)$ denotes the limit of $g(p)$ as $p$ tends to $s$ in $\C_{i_s}$.
\end{definition}
\begin{corollary}
If $f\in\lhol(U)$ then $\sderiv f \in \lhol(U)$ and $\sderiv f (s) = \frac{\partial}{\partial s_0} f(s)$. 
\end{corollary}
\begin{remark}
Observe that $\sderiv f(s)$ is well defined: since it coincides with the derivative with respect to the real part of the variable, it does not depend on the choice of the imaginary unit $i_s$ in \eqref{sderiv} if $s$ is real.
\end{remark}
We also adopt the following notation:
\begin{definition}
Let $f$ be a function defined on a set $U\subset\H$  and let $i\in\S$. We denote the restriction of $f$ to the complex plane $\C_i$ by $f_i$, i.e. $f_i := f|_{U\cap \C_i}$.
\end{definition}
\begin{remark}
In this paper, a subscript $i$ always indicates the restriction of a function, set or variable to the complex plane $\C_i$ that is determined by the imaginary unit $i$. In order to avoid confusion, we do not denote indices in sums etc. by the symbol $i$ (neither by $j$ or $k$, which also refer to imaginary units). Moreover, we denote the imaginary unit in the usual complex field by $\imath$ in order to distinguish it from imaginary units in the field of quaternions.
\end{remark}
Using the above notation we can formulate the so called Splitting Lemma.
\begin{lemma}[Splitting Lemma]\label{SplitLem}
Let $U\subset\H$ be axially symmetric and let $i,j\in\S$ with $i\perp j$.
 If $f\in\lhol(U)$, then there exist holomorphic functions
$f_1,f_2: U\cap\C_i\to\C_i$ such that $ f_i = f_1 + f_2j$.
\end{lemma}
Another important result is the representation formula, which allows to reconstruct the function $f$ from its values on a single complex plane.
\begin{theorem}[Representation Formula]\label{RepFo}
Let $U\subset\H$ be axially symmetric and let $i\in\S$. For any $x = x_0 + i_x x_1\in U$ set $x_i := x_0 + ix_1$. If $f\in\lhol(U)$. Then
\begin{equation}\label{Reppp}
f(x) = \frac12(1-i_xi)f(x_i) + \frac12(1+i_xi)f(\overline{x_i}).
\end{equation}
\end{theorem}

As a consequence of the representation formula, every $\H$-valued holomorphic function on $\C_i$ that is defined on a suitable domain has a slice hyperholomorphic extension.
\begin{lemma}[Extension Lemma]\label{ExtLem}
Let $O\subset \C_i$ be symmetric with respect to the real axis and let $[O]$ be its axially symmetric hull, i.e.
\[ [O] := \{ x_0 + i_x x_1: x_0 + i x_1\in O, i_x\in\S\}.\]
If $f: O \to \H$ satisfies
\begin{equation}\label{SHOr}
 \frac{1}{2}\left(\frac{\partial}{\partial x_0}f(x)  + i \frac{\partial}{\partial x_1}f(x)\right) = 0,
\end{equation}
then there exists a unique slice hyperholomorphic extension $\ext(f)$ of $f$ to $[O]$.
\end{lemma}
\begin{remark}
The slice hyperholomorphic extension of a function $f$ is obviously given by \eqref{Reppp}. If however $f$ is a power series $f(z) = \sum_{n=0}^{+\infty}z^na_n$ with $a_n\in\H$, then $\ext(f)$ is simply obtained by extending the variable from a complex to a quaternionic one, i.e. $\ext(f)(x) = \sum_{n=0}^{+\infty}s^na_n$.
\end{remark}
\begin{remark}\label{PolyExt}
Originally, slice hyperholomorphic functions were defined as functions that satisfied \eqref{SHOr} for all $i\in\S$. With this definition, the representation formula holds true only on axially symmetric slice domains, where it implies that such functions can be represented in the form \eqref{LSHol}. Most important results for slice hyperholomorphic functions are however based on the representation formula such that the theory was often only developed on axially symmetric slice domains. \Cref{LHolDef} in contrast includes in general less functions, but it allows to develop the theory also for functions, whose domains are not connected or do not intersect the real line. This is in particular important in operator theory. On axially symmetric slice domains both definitions are equivalent.
\end{remark}

Finally, we recall some results on Hardy spaces of slice hyperholomorphic functions from \cite{HardyRef}.
\begin{definition}
Let $0<p<+\infty$ and let $\D$ denote the unit ball in $\H$. The Hardy space $\hardy^p(\D)$ of slice hyperholomorphic functions on $\D$ consists of all functions $f\in\lhol(\D)$ such that 
\begin{equation}\label{hardyDef}
\| f\|_{\hardy^p(\D)} :=\sup_{i\in\S} \lim_{r\nearrow 1} \left(\frac{1}{2\pi}\int_{0}^{2\pi} \left| f\left(re^{i\theta}\right)\right|^p\,d\theta\right)^{\frac{1}{p}} < +\infty. 
\end{equation}
\end{definition}
For $1\leq p < + \infty$, the space $\hardy^p(\D)$ is a quaternionic Banach space. Moreover, due to the representation formula, it is sufficient that  
\[\lim_{r\nearrow 1} \left(\frac{1}{2\pi}\int_{0}^{2\pi} \left| f\left(re^{i\theta}\right)\right|^p\,d\theta\right)^{\frac{1}{p}} < +\infty\]
 for one $i\in\S$ in order to have $f\in\hardy^p(\D)$. Finally, a slice hyperholomorphic function belongs to this space if and only if its components obtained from the Splitting Lemma belong to the respective complex hardy space as the following lemma shows.
\begin{lemma}\label{HardySplit}
Let $f\in\lhol(\D)$, let $i,j\in\S$ with $i\perp j$ and write $f_i = f_1 + f_2j$ with holomorphic components $f_1,f_2$ according to \Cref{SplitLem}. For $0<p<+\infty$, we have $f\in\hardy^p(\D)$ if and only if $f_1,f_2$ belong to the complex Hardy space $\hardy^p_{\C}(\D_i)$ of holomorphic functions on the unit ball $\D_i$ in the complex plane $\C_i$.
\end{lemma}
One moreover has for $p\geq 1$ that
\begin{equation}\label{HardySplitEst}
\| f_{\ell}\|_{\hardy_{\C}^p(\D_i)} \leq \|f\|_{\hardy^{p}(\D)} \lesssim \| f_{1}\|_{\hardy_{\C}^p(\D_i)}+ \| f_{1}\|_{\hardy_{\C}^p(\D_i)}.
\end{equation}

\section{Carleson measures on $\D$}\label{CarlMeasSec}
We introduce now the concept of Carleson measures on $\D$ and discuss them in the slice hyperholomorphic setting. These measures were introduced in \cite{Arcozzi} in order to study Hankel operators on the Hardy space $\hardy^2(\D)$. They were also studied in the context of general Hardy and Bergman spaces on $\D$ in \cite{sabsar}.
\begin{definition}
For $h>0$, $\theta_0\in[0,2\pi)$ and $i\in\S$ let $S_i(\theta_0,h)$ be the Carleson box in the plane $\C_i$ defined by
\[ S_i(\theta_0,h) := \left\{ re^{i\theta}\in\D_i\ :\ |\theta-\theta_0|\leq h, 1-h\leq r\leq 1\right\}.\]
The set
\[ S(\theta_0,h) := \bigcup_{i\in\S}S_i(\theta_0,h)\] 
is called a symmetric box.
\end{definition}

\begin{definition}
A nonnegative Radon measure $\mu$ on $\D$ is called a Carleson measure if there exists a constant $C>0$ such that  for all $h>0$ and all $ 0\leq \theta_0 \leq \pi$
\begin{equation}\label{CB}
 \mu(S(\theta_0,h)) \leq C h.
\end{equation}
\end{definition}
We recall an observations made in \cite{Arcozzi}: a finite Radon measure $\mu$ on the unit ball $\D$ can be decomposed uniquely as $\mu = \mu_\R + \tilde{\mu}$ such that $\supp(\mu_\R)\subset \D\cap\R$ and $\tilde{\mu}(\D\cap\R) = 0$. If $\nu$ is the measure on the Borel sets $\borel(\S)$ of the sphere $\S$ that is defined by 
\[\nu(E) := \tilde{\mu}\left(\{ x_0+ix_1 \in \D_i^+: i\in E\}\right)\qquad \forall E\in\borel(\S),\]
with $\D_i^+ = \{x_0 + i x_1\in\D: x_1> 0\}$, then the Disintegration Theorem \cite[Theorem~2.28]{DisIntBook} implies the existence of a family of probability measures $(\mu_i^+)_{i\in\S}$ such that $\mu_i^+$ is defined on $\D_i^+$ and such that 
\[ d\tilde{\mu}(x_0 + i x_1) = d\mu_i^+(x_0+ix_1)d\nu(i)\]
that is
\[ \int_{\D} f(x) \,d\tilde{\mu}(x) = \int_{\S}\left(\int_{\D_i^+} f(x)\,d\mu_i^+(x)\right)d\nu(i) \]
for all $f\in L^1(\D, \tilde{\mu})$. With this notation we have
\begin{equation}\label{BigIntID}
 \int_{\D} f(x) \, d\mu(x) = \int_{\R\cap\D} f(x)\, d\mu_{\R}(x) + \int_{\S}\left(\int_{\D_i^+}f(x)\,d\mu_i^{+}(x)\right)d\nu(i).
 \end{equation}
Observe that each of the measures $\mu_i^+$ defines a measure on the entire ball $\D_i$, which is obtained by setting
\[ \mu_i^+(E) := \mu_i^+(E\cap\D_i^+) \qquad\forall E\in\borel(\D_i).\]

\begin{definition}
Let $\mu$ be a nonnegative Radon measure on $\D$. A slice-de\-com\-position of $\mu$ is a triple $\left(\mu_{\R},\nu,(\mu_i^+)_{i\in\S}\right)$ consisting of a measure $\mu_\R$ on $\D$ with $\supp(\mu_{\R})\subset\R$, a measure $\nu$ on $\S$ and a family of probability measures $(\mu_{i}^+)_{i\in\S}$ such that $\mu_i^+$ is defined on $\D_i^+$ and such that $\mu$ can be decomposed as 
\[d\mu(x_0+ix_1)  = d\mu_{\R}(x_0+ix_1) + d\mu_{i}^+(x_0+ix_1)d\nu(i)\]
 as in the above discussion.
\end{definition}
\begin{remark}
We point out that the measures $\mu_i$ are not unique: if $\hat{\mu}_i^{+}$ is another family of measures with the above properties, then $\mu_i^{+} = \hat{\mu}_i^{+}$ only for $\nu$-allmost all $i$. Hence, we can replace $\mu_i^+$ by an arbitrary probability measure on $\D_i^+$ if $\nu(i) = 0$.
\end{remark}

\begin{remark}\label{ArcozziRem}
In \cite{Arcozzi} the authors do an additional step: they set  $\mu_{i} := \mu_{i}^+ + \mu_{-i}^+$ for each $i\in\S$ and obtain in this way a measure define on the entire unit ball $\D_i$ in $\C_i$. Moreover $\mu_{i} = \mu_{-i}$. Thus they write \eqref{BigIntID} as
\[ \int_{\D} f(x) \, d\mu(x) = \int_{\R\cap\D} f(x)\, d\mu_{\R}(x) + \frac{1}{2}\int_{\S}\left(\int_{\D_i}f(x)\,d\mu_i(x)\right)d\nu(i).\]
This formula holds if $\nu$ is invariant under the transformation $i\mapsto -i$ because then
\begin{align*}
&\int_{\S}\left(\int_{\D_i^+}f(x)\,d\mu_i^{+}(x)\right)d\nu(i) \\
=& \frac{1}{2}\int_{\S}\left(\int_{\D_i^+}f(x)\,d\mu_i^{+}(x)\right)d\nu(i) + \frac{1}{2}\int_{\S}\left(\int_{\D_{-i}^+}f(x)\,d\mu_{-i}^{+}(x)\right)d\nu(-i)\\
=&\frac{1}{2}\int_{\S}\left(\int_{\D_i^+}f(x)\,d\mu_i^{+}(x)\right)d\nu(i) + \frac{1}{2}\int_{\S}\left(\int_{\D_{-i}^+}f(x)\,d\mu_{-i}^{+}(x)\right)d\nu(i) \\
=& \frac{1}{2}\int_{\S}\left(\int_{\D_i}f(x)\,d\mu_i(x)\right)d\nu(i).
\end{align*}
If however $\nu$ is not invariant under the transformation $i\mapsto -i$ this formula does not hold. An easy counterexample is the measure $\mu = \delta_{a}$ with $a\in\D$ not real and $\delta_a$ denoting the Dirac measure at $a$. A slice decomposition of $\mu$ is then given by $\mu_{\R} = 0$, $\nu(i) = \delta_{i_a}$ and $\mu_i^+= \delta_{a}$ if $i = i_{a}$ and $\mu_{i}^+ = \delta_{b_i}$ with $b_i \in \D_i^+$ arbitrary if $i\neq i_a$. Then
\[ \int_{\S}\left(\int_{\D_{i}^+} f(x)\,d\mu_{i}^+(x)\right) d\nu(i) = \int_{\D_{i_a}^+} f(x)\,d\delta_{a}(x) = f(a)\]
but
\begin{gather*}
\frac{1}{2}\int_{\S}\left(\int_{\D_{i}} f(x)\,d\mu_{i}(x)\right) d\nu(i) \\
=  \frac{1}{2} \int_{\D_{i_a}} f(x) d(\delta_{a} + \delta_{b_{-i_{a}}}) = \frac{1}{2} f(a) + \frac{1}{2} f(b_{-i_{a}}). 
\end{gather*}
In this paper, we therefore chose to work directly with \eqref{BigIntID} in order to obtain more general statements.
\end{remark}

\begin{definition}
A slice-decomposition of a finite Radon measure $\mu$ is called a slice Carleson decomposition if the measures $\mu_i^+$ for $i\in\S$ and the measure $\mu_{\R}$ are Carleson measures on $\D_i$ with a common Carleson bound, that is if there exists a constant $C>0$ such that for all $h>0$ and all $0\leq \theta_0\leq \pi$ and all $i\in\S$
\begin{equation}\label{SliceCE}
 \mu_{\R}\left(S_i(\theta_0,h)\right) \leq C h \quad\text{and}\quad \mu_i^+\left(S_i(\theta_0,h)\right) \leq C h.
\end{equation}
\end{definition}

\begin{corollary}\label{SC->C}
If a finite Radon measure $\mu$ on $\D$ has a slice Carleson decomposition, then it is a Carleson measure.
\end{corollary}
\begin{proof}
If $(\mu_{\R}, \nu, (\mu_{i}^{+})_{i\in\S})$ is a slice Carleson decomposition of $\mu$, then we have for $\theta_0\in[0,\pi]$ and $h>0$  by \eqref{BigIntID} and \eqref{SliceCE} that
\begin{gather*}
 \mu\left(S\left(\theta_0,h\right)\right) = \mu_{\R}(S_{i_0}(\theta_0,h)) + \int_{\S} \mu_{i}^{+}\left(S_i(\theta_0,h)\right)\, d\nu(i) \\
 \leq C h + \int_{\S} Ch \,d\nu(i) =  C  (1 + \nu(\S)) h,
 \end{gather*}
where $i_0$ is an arbitrary imaginary unit in $\S$. Hence, $\mu$ is a Carleson measure.

\end{proof}

\begin{example}
Let $i_{n}\in\S$ for $n\in\N$ such that $i_m \neq i_n$ for $n\neq m$ and define 
\[\mu := \sum_{n=0}^{+\infty}n^{-\frac{3}{2}} \delta_{a_n}\quad\text{with $a_n=\frac{n-1}{n}i_n$},\]
 where $\delta_a$ denotes the point measure at a point $a$. Obviously $\mu$ is a finite Radon measure on $\D$ with $\mu(\D) = \sum_{n=0}^{+\infty}n^{-\frac{3}{2}} < +\infty$. A slice decomposition of $\mu$ is given by $\mu_{\R} = 0$ and $\nu = \sum_{n=0}^{+\infty}n^{-\frac{3}{2}}\delta_{i_n}$ and $\mu_{i}^{+}=\delta_{a_n}$ if $i = i_n$ for some $n\in\N$ and $\mu_{i}^{+} = \delta_{\frac{1}{2}i}$ otherwise. 

Each of the measures $\mu_i^+$ is a Carleson measure on $\D_i$: for $i\neq i_n, n\in\N$ one can choose $C = 2$ in \eqref{SliceCE} and for $i_n$ the Carleson bound is $C=n$. This however shows that it is not possible to find a global slicewise Carleson bound $C$ such that \eqref{SliceCE} holds for all constants $i$ in $\S$. Thus the considered slice decomposition is not Carleson. (Since the measures $\mu_{i_n}^+$, which create the problems, are uniquely defined because $\nu(i_n)\neq 0$, this holds even for any slice decomposition of~$\mu$.)

Finally, for $h = 1/n$, we have
\begin{gather*}
 \frac{\mu\left(S\left(\frac{\pi}{2}, \frac{1}{n}\right)\right)}{\frac{1}{n}} = \sum_{k=n}^{+\infty} k^{-\frac{3}{2}} n =  \sum_{\ell=1}^{+\infty}\sum_{k=0}^{n-1}(\ell n + k)^{-\frac{3}{2}} n\\ 
\geq \sum_{\ell=1}^{+\infty}\sum_{k=0}^{n-1}((\ell+1) n)^{-\frac{3}{2}} n =  n^{\frac{1}{2}}\sum_{\ell=2}^{+\infty}\ell^{-\frac{3}{2}}.
\end{gather*}
Since the last expression tends to infinity as $n$ tends to infinity, it is obviously not possible to find a bound such that \eqref{CB} holds true. Thus the measure $\mu$ is not a Carleson measure on $\D$ although all of the measures $\mu_i^+$ are Carleson measures on $\D_i^+$. We therefore find that the condition of uniformity in \eqref{SliceCE} cannot be relaxed without losing the validity of \Cref{SC->C}. 
\end{example}
\begin{example}\label{ExFMeas}
Let $\lambda_4$ denote the Lebesgue measure on $\H \cong  \R^4$, let $\rho\in L^1(\D, \lambda_4)$ with $\rho(s)> 0$ almost everywhere and let $\mu(s) = \rho(s) d\lambda_4(s)$ be the measure with density $\rho$ with respect to $\lambda_4$. If we write the vectorial part of the variable $s$ in spherical coordinates, we have for any function $f\in L_1(\D,\mu)$ that
\begin{gather*}
\int_{\D}f(s)\,d\mu(s) = \int_{\D}f(s)\rho(s)\,d\lambda_4(s) \\
= \int_{0}^{2\pi}\int_{0}^{\pi}\int_{\D_i^+} f(s_0 + i(\theta,\varphi)s_1)\rho(s_0+i(\theta,\varphi)s_1)s_1^2\sin(\theta)\,d\lambda_2(s_0,s_1)\,d\theta\,d\varphi,
\end{gather*}
where $i(\theta,\varphi) = \sin(\theta)\cos(\varphi) e_1 + \sin(\theta)\sin(\varphi) e_2 + \cos(\theta)e_3$ and $\lambda_2$ denotes the two-dimensional Lebesgue measure. Hence,
\begin{equation}
\int_{\D}f(s)\,d\mu(s) 
= \int_{\S}\left(\int_{\D_i^+} f(s_0 + i s_1)\rho(s_0+ i s_1 )s_1^2\,d\lambda_2(s_0,s_1)\right) d\sigma(i),
\end{equation}
where $\sigma$ denotes the surface measure on $\S$. After setting
\[ C(i) = \int_{\D_i} \rho(s_0 + is_1) s_1^2 \,d\lambda_2(s_0,s_1)\]
a slice decomposition of $\mu$ is therefore given by
\begin{align*}
\displaystyle\mu_\R &= 0 \\
\displaystyle d\nu(i) &= C(i)\,d\sigma(i)\\
\displaystyle d\mu_i^+(s)& = \frac{1}{C(i)}\rho(s)s_1^2\,d\lambda_2(s).
\end{align*}
\end{example}

\begin{definition}
We call a Carleson measure $\mu$ a vanishing Carleson measure if uniformly for $\theta_0\in[0,\pi]$  we have
\begin{equation}\label{VCB}
 \lim_{h\to 0}\frac{\mu(S(\theta_0,h))}{h} = 0 . 
 \end{equation}
\end{definition}
\begin{definition}
We say that a slice-decomposition $(\mu_\R, \nu, (\mu_i^+)_{i\in\S})$ of a measure $\mu$ is called a vanishing slice Carlson decomposition if the measures $\mu_\R$ and $\mu_i$ are vanishing Carleson measures, that is if  one has uniformly in $\theta_0$ and $i\in\S$ that
\[ \lim_{h\to0} \frac{\mu_{\R}(S_i(\theta_0,h))}{h} = 0\quad \text{and}\quad \lim_{h\to 0}\frac{\mu_i^+(S_i(\theta_0,h))}{h} = 0.\]

\end{definition}

\begin{corollary}\label{VSC->VC}
If a finite Radon measure $\mu$ on $\D$ has a vanishing slice Carleson decomposition, then it is a vanishing Carleson measure.
\end{corollary}
\begin{proof}
If $(\mu_{\R},\nu,(\mu_i)_{i\in\S})$ is a vanishing slice Carleson decomposition of $\mu$, then there exists a function $\mathcal{E}(h)$ such that $\mathcal{E}(h)\geq 0$ and $\mathcal{E}(h)\to 0$ as $h\to 0$ and such that $\mu_{\R}(S_i(\theta_0,h)) \leq \mathcal{E}(h)h$ and $\mu_{i}(S_i(\theta_0,h)) \leq \mathcal{E}(h)h$ for any $i\in\S$ and any $\theta_0\in[0,\pi]$. Thus, after choosing an arbitrary imaginary unit $i_0\in\S$, we have 
\begin{gather*}
\mu(S(\theta_0,h)) = \mu_{\R}(S_{i_0}(\theta_0,h)) + \int_{\S}\mu_i^+(S_i(\theta_0,h))\,d\nu(i) \\
 \leq \mathcal{E}(h)h + \int_{\S}\mathcal{E}(h)h\,d\nu(i)  \leq (1+ \nu(\S))\mathcal{E}(h)h
\end{gather*}
for any $\theta_0\in[0,\pi]$. Hence \eqref{VCB} holds true uniformly in $\theta_0$ and $\mu$ is therefore a vanishing Carleson measure.

\end{proof}

\section{The space $\BMOSH(\D)$}\label{BMOSect}
In the complex setting, the space $\BMOA(\D_{\C})$ of holomorphic functions of bounded mean oscillation on the complex unit disc $\D_{\C}$ consists of all functions in the complex Hardy space  $\hardy^2(\D_{\C})$ such that
\[ \| f \|_{\BMOA(\D_{\C})} := |f(0)| + \sup_{I \subset \R, |I|\leq 2\pi} \sqrt{ \frac{1}{|I|} \int_{I}\left| f\left(e^{\imath\theta}\right) - f_{I}\right|^2 \,d\theta }< +\infty,\]
where $I=(\alpha,\beta)\subset \R$ denotes an interval,  $|I| = \beta - \alpha$ is length of the interval and $f_I := \frac{1}{|I|}\int_{I} f\left(e^{\imath\theta}\right) \,d\theta$. This motivates the following definition.
\begin{definition}
Let $f$ be a function in $\hardy^2(\D)$. For any interval $I =(\alpha,\beta)\subset\R$ with $|I |:= \beta-\alpha \leq 2\pi$ and $i\in\S$, we denote by $f_{I ,i}$ the average value (of the radial limit) of $f$ on the arc $\left(e^{\alpha i},e^{\beta i}\right) \subset \partial\D_i$, i.e.
\[ f_{I ,i} = \frac{1}{|I |}\int_{I }f\left(e^{i\theta}\right)\,d\theta.\]
We say that the function $f$ belongs to $\BMOSH(\D_i)$ for some $i\in\S$ if 
\[ |f|_{\BMOSH(\D_i)}^2 := \sup_{I \subset\R, |I |\leq 2\pi} \frac{1}{|I |}\int_{I }\left|f\left(e^{i\theta}\right) - f_{I ,i}\right|^2\,d\theta<+\infty\]
and to $\BMOSH(\D)$ if 
\[ | f|_{\BMOSH(\D)} := \sup_{i\in\S}|f|_{\BMOSH(\D_i)}<+\infty.\]
\end{definition}

Arcozzi and Sarfatti studied first properties of these spaces in \cite{Arcozzi} in order to investigate Hankel operators in a quaternionic setting. In particular they showed the following lemma, which is a direct consequence of the representation formula and as a consequence of which we will only consider the space $\BMOSH(\D)$ in the following.
\begin{lemma}\label{BMOSplit}
Let $i\in\S$. Then $f\in\BMOSH(\D_i)$ if and only if $f\in\BMOSH(\D)$ and
\[ |f|_{\BMOSH(\D_i)} \leq |f|_{\BMOSH(\D)} \leq 2 |f|_{\BMOSH(\D_i)}.\]
\end{lemma}
They also showed in the proof of \cite[Theorem~5.3]{Arcozzi} the following consequence of the splitting lemma.
\begin{lemma}\label{BMOComponents}
Let $f\in\lhol(\D)$, let $i,j\in\S$ with $i\perp j$ and write $f_i = f_1+f_2j$ with holomorphic functions $f_\ell: \D_i\to\C_i$ according to \Cref{SplitLem}.  Then $f\in\BMOSH(\D)$ if and only if $f_1,f_2\in\BMOA(\D_i)$.
\end{lemma}

\begin{lemma}\label{BMONorms}
For any $i\in\S$, the space $\BMOSH(\D)$ equipped with the norm
\[ \|f\|_{\BMOSH(\D_i)}  = |f(0)|+ |f|_{\BMOSH(\D_i)}\]
is a quaternionic right Banach space. Similarly, the space $\BMOSH(\D)$ is a quaternionic right Banach space, when it is equipped with the norm
\[ \|f \|_{\BMOSH(\D)} = |f(0)| + |f|_{\BMOSH(\D)}.\]
Moreover, these norms are equivalent.
\end{lemma}
\begin{proof}
It is immediate that $\|\cdot\|_{\BMOSH(\D_i)}$ and $\|\cdot\|_{\BMOSH(\D)}$ define norms on $\BMOSH(\D)$. Moreover, these norms are equivalent because of \Cref{BMOSplit}.

Let us now discuss the completeness of $\BMOSH(\D)$ with respect to these norms: let $i,j\in\S$ with $i\perp j$ and write $f_i = f_1 + f_2 j$ for $f\in\BMOSH(\D)$ according to \Cref{SplitLem}. By \Cref{BMOSplit}, $f\in\BMOSH(\D)$ if and only if  $f_1,f_2\in \BMOA(\D_i)$.

We establish a relation between the norm $\|f\|_{\BMOSH(\D_i)}$ of $f$ and the norms $\|f_1\|_{\BMOA(\D_i)}$ and $\|f_2\|_{\BMOA(\D_i)}$ of $f_1$ resp. $f_2$. First observe that $a \in\H$ can be written as $a = a_1 + a_2j$ with $a_1,a_2\in\C_i$ and that in this case 
\begin{equation}\label{aai}
|a|^2 = |a_1|^2 + |a_2|^2
\end{equation}
 and in turn for any interval $I\subset \R$ with $|I|\leq 2\pi$
\begin{equation}\label{AGH}\begin{split}
& \frac{1}{|I|} \int_{I} |f(e^{i\theta}) - f_{I,i}|^2\,d\theta \\
=&\frac{1}{|I|} \int_{I} |f_1(e^{i\theta}) - f_{I,1}|^2\,d\theta +  \frac{1}{|I|} \int_{I} |f_2(e^{i\theta}) - f_{I,2}|^2\,d\theta,
\end{split}\end{equation}
where we set
\[f_{I,\ell} = \frac{1}{|I |}\int_{I }f_{\ell}\left(e^{i\theta}\right)\,d\theta\qquad\text{for $\ell\in\{1,2\}.$}\]
 Since for $a,b >0$ the inequality \(\sqrt{a+b} \leq \sqrt{a} + \sqrt{b}\) holds true, we deduce from \eqref{aai} first that $|a| \leq |a_1| + |a_2|$ for $a = a_1 + a_2j$. From this estimate and~\eqref{AGH} we conclude that 
 \begin{align*}
 \|f\|_{\BMOSH(\D_i)} =& |f(0)| + \sup_{I\subset\R, |I|\leq 2\pi } \sqrt{\frac{1}{|I|}\int_{I}\left|f\left(e^{i\theta}\right) - f_{I,i}\right|^2\,d\theta}\\
 \leq & |f_1(0)| + \sup_{I\subset\R, |I|\leq 2\pi } \sqrt{\frac{1}{|I|}\int_{I}\left|f_1\left(e^{i\theta}\right) - f_{I,1}\right|^2\,d\theta}\\
& + |f_2(0)| + \sup_{I\subset\R, |I|\leq 2\pi } \sqrt{\frac{1}{|I|}\int_{I}\left|f_2\left(e^{i\theta}\right) - f_{I,2}\right|^2\,d\theta}\\
= & \|f_1\|_{\BMOA(\D_i)} + \|f_2\|_{\BMOA(\D_i)}.
 \end{align*}
On the other hand \eqref{aai} implies for $\ell\in\{1,2\}$ that $|a| = |a_1 + a_2j | \geq |a_\ell|$ and hence
 \begin{align*}
 \|f_\ell\|_{\BMOA(\D_i)} = & |f_\ell(0)| + \sup_{I\subset\R, |I|\leq 2\pi } \sqrt{\frac{1}{|I|}\int_{I}\left|f_1\left(e^{i\theta}\right) - f_{I,1}\right|^2\,d\theta}\\
& \leq |f(0)| + \sup_{I\subset\R, |I|\leq 2\pi } \sqrt{\frac{1}{|I|}\int_{I}\left|f\left(e^{i\theta}\right) - f_{I,i}\right|^2\,d\theta}\\
= & \|f\|_{\BMOSH(\D_i)}.
 \end{align*}
Altogether we have
\begin{equation}\label{BMOiSplitEst}
\|f_{\ell}\|_{\BMOA(\D_i)}\leq \|f\|_{\BMOSH(\D_i)} \leq \|f_1\|_{\BMOA(\D_i)} + \| f_2\|_{\BMOA(\D_i)}.
\end{equation}
By the equivalence of the norms, we also obtain
\begin{equation}\label{BMOSplitEst}
\|f_{\ell}\|_{\BMOA(\D_i)}\lesssim \|f\|_{\BMOSH(\D)} \lesssim \|f_1\|_{\BMOA(\D_i)} + \| f_2\|_{\BMOA(\D_i)}.
\end{equation}
The completeness of $\BMOSH(\D)$ follows now from the completeness of the complex space $\BMOA(\D_i)$, cf. \cite[Theorem~5.1]{Girela}. If $(f_n)_{n\in\N}$ is a Cauchy sequence in $\BMOSH(\D)$ w.r.t. $ \| \cdot \|_{\BMOSH(\D_i)}$, then $(f_{n,1})_{n\in\N}$ and $(f_{n,2})_{n\in\N}$ are Cauchy sequences in $\BMOA(\D_i)$ because of \eqref{BMOiSplitEst}. Hence, they converge to two functions $f_1$ and $f_2$ in $\BMOA(\D_i)$. By \Cref{BMOSplit}, we have $f:= \ext(f_1 + f_2j) \in \BMOSH(\D)$ and $f_n \to f$ w.r.t. $\|\cdot \|_{\BMOSH(\D_i)}$ because of \eqref{BMOiSplitEst}. Thus $\BMOSH(\D_i)$ is complete w.r.t. $\|\cdot \|_{\BMOSH(\D_i)}$ and because of the equivalence of the norms also w.r.t. $\| \cdot\|_{\BMOSH(\D)}$.

\end{proof}

\begin{remark}
The space $\BMOSH(\D)$ is not separable, which follows easily from the fact that the classical space $\BMOA(\D_i)$ is not separable. Indeed, $\BMOA(\D_i)$ can be embedded into $\BMOSH(\D)$ by mapping $f\in\BMOA(\D_i)$ to its slice hyperholomorphic extension $\ext(f)\in\BMOSH(\D)$. In this case,
 $\| f\|_{\BMOA(\D_i)} = \| \ext(f)\|_{\BMOSH(\D_i)}$. The functions
 \[ f_{\alpha,i}(z) = \log\left(\frac{1}{1 - e^{i\alpha}z}\right) \qquad \alpha \in[0,2\pi),\] 
 constitute an uncountable family of functions in $\BMOA(\D_i)$ such that the distances $\| f_{\alpha,i} - f_{\beta,i}\|_{\BMOA(\D_i)}$ are bounded by a positive constant from below, cf. \cite[Corollary~5.4]{Girela}. Hence, the family $(f_\alpha)_{\alpha\in [0,2\pi)}$ defined by
 \begin{align*}
    &f_{\alpha}(s)= \ext(f_{\alpha,i})(s) \\
    =& \frac{1}{2}(1-i_si)   \log\left(\left(1 - e^{i\alpha}s_i\right)^{-1}\right)   +  \frac{1}{2}(1+  i_si)   \log\left(\left(1 - e^{i\alpha}\overline{s_i}\right)^{-1}\right). 
    \end{align*}
with $s_i = s_0 + i s_1$ for $s = s_0+i_ss_1$ is an uncountable family of functions in $\BMOSH(\D)$ with the same property. 
\end{remark}

As in the classical setting, the space $\BMOSH(\D)$ is invariant under (appropriate) M\"obius transformations.
\begin{definition}For $a \in \D_i$, let $T_a$  denote the slice hyperholomorphic M\"obius transformation
\begin{equation}
T_a(s) = \left(1-sa_0 + s^2 |a|^2\right)^{-1}\left(s^2a - s\left(1+|a|^2\right)-\overline{a}\right).
\end{equation}
\end{definition}
\begin{remark}
The slice hyperholomorphic M\"obius transformation $T_a$ is obtained by extending the complex M\"obius transform \(T_{a,\C}(z):=(z+a)/(1+\overline{a}z)\)
 via the extension lemma to all of $\D$. In terms of the slice hyperholomorphic product $\ast$, which preserves slice hyperholomorphicity, it can be expressed as $T_a(s) = (1-s\overline{a})^{-\ast}\ast(q+\overline{a})$. For further details, we refer to \cite{SRMoeb}, where slice hyperholomorphic M\"obius transformations were introduced.
\end{remark}
\begin{definition}
For $f,g\in\lhol(\D)$ with $g(\D_i)\subset \D_i$, we define the $i$-composition of $f$ and $g$ as
\[ f\circ_i g := \ext(f_i\circ g_i).\]
\end{definition}
\begin{remark}
If $a\in\D_i$, then obviously
\[ f\circ_i T_a = \ext(f_i\circ T_{a,\C}).\]
The definition of the $i$-composition is necessary because the usual composition does not preserve slice hyperholomorphicity. If however $g$ is intrinsic, in particular if $g=T_a$ with $a\in\R$, then the usual composition preserves slice hyperholomorphicity and it agrees with the $i$-composition for any $i\in\S$.
\end{remark}

\begin{proposition}
Let $i\in \S$ and $a\in \D_i$ and consider the slice hyperholomorphic M\"obius transformation $T_{a}$. If  $ f\in \BMOSH(\D)$  then $f\circ_i  T_{a}\in \BMOSH(\D)$.
\end{proposition}
\begin{proof}
Let $j\in\S$ with $i\perp j$ and apply \Cref{SplitLem} in order to write  $f_i =f_1+f_2 j$ with holomorphic functions $f_1,f_2: \D_i\to\C_i$.  Because of \Cref{BMOComponents}, the functions $f_1$ and $f_2$ belong to $\BMOA(\D_i)$ and the classical theory implies that $f_1\circ T_a$ and $f_2\circ T_a$ belong to $\BMOA(\D_i)$ too, cf. \cite[p. 22]{danikas}. Finally, as 
\[(f\circ_i T_{a})_i =  f_1\circ T_{a} +  f_2\circ T_{a}  j\]
the desired result follows again from \Cref{BMOComponents}.

\end{proof}

As in the complex case, it is possible to define a semi-norm on $\BMOSH(\D)$ that is invariant under certain M\"obius transformations.

\begin{lemma}
For $i\in\S$, the function given by
\[ |f|_{*,i}^2 = \sup_{a\in\D_i}\frac{1}{2\pi} \int_{-\pi}^{\pi}\left|f\circ_i T_a\left(e^{i\theta} \right) - f(a)\right|^2\,d\theta \]
is a semi-norm on $\BMOSH(\D)$ that satisfies 
\begin{equation}\label{MoebInv}
| f\circ_i T_a |_{*,i} = | f |_{*,i}  \qquad\text{for }a\in\D_i. 
\end{equation}
Moreover, the norm
\[  \| f\|_{\ast,i} := |f(0)| + |f|_{\ast,i}\]
is equivalent to $\|\cdot\|_{\BMOSH(\D)}$.
\end{lemma}
\begin{proof}
It is immediate that $|f|_{\ast,i} = 0$ if and only if $f$ is constant, that $|f a|_{\ast,i} = |f|_{\ast,i}|a|$ and $|f + g|_{\ast,i} \leq |f|_{\ast,i} + |g|_{\ast,i}$ for $a\in\H$ and $f,g\in\BMOSH(\D)$ and that \eqref{MoebInv} holds true. Moreover, if we choose $j\in\S$ with $i\perp j$ and write $f_i = f_1 + f_2j$ with holomorphic components according to \Cref{SplitLem}, then $|f|_{*,i}^2 \leq |f_1|_{*,\C}^2 + |f_2|_{*,\C}^2$, where 
\[|f_{\ell}|_{*,\C} = \sup_{a\in\D_i}\frac{1}{2\pi} \int_{-\pi}^{\pi}\left|f_{\ell}\circ T_{a,\C}\left(e^{i\theta} \right) - f(a)\right|^2\,d\theta\]
 denotes the M\"obius invariant semi-norm on $\BMOA(\D_i)$. Hence, $|f|_{\ast,i}<+\infty$ for $f\in\BMOSH(\D)$. Moreover, after setting $\| f _{\ell}\|_{\ast,\C} := |f_{\ell}(0)| + |f_{\ell}|_{\ast,\C}$, the equivalence of $\|\cdot\|_{\ast,\C}$ and $\|\cdot\|_{\BMOA(\D_i)}$ in the complex case (cf. Theorem~5.1 in \cite{Girela}) and \eqref{BMOSplitEst} imply by computations as in the proof of \Cref{BMONorms} that 
\begin{gather*}
\| f \|_{\ast,i} \leq \|f_1\|_{\ast,\C} + \|f_2\|_{\ast,\C} \lesssim \| f_1 \|_{\BMOA(\D_i)} + \| f_2 \|_{\BMOA(\D_i)} \\
\lesssim \| f \|_{\BMOSH(\D)} \lesssim \|f_1\|_{\BMOA(\D_i)} + \|f_2\|_{\BMOA(\D_i)} \\
\lesssim \|f_1\|_{\ast,\C} + \| f_2\|_{\ast,\C} \lesssim \|f \|_{\ast,i},
\end{gather*}
where all constants are independent of $f$, $f_1$ and $f_2$. Hence, $\|\cdot\|_{\BMOSH(\D_i)}$ and $\|\cdot\|_{\ast,i}$ are equivalent norms.

\end{proof}

\begin{remark}
It is possible to define a slice-independent semi-norm by setting
\[ |f|_{\ast} := \sup_{i\in\S} |f|_{*,i} = \sup_{a\in\D}\sqrt{\frac{1}{2\pi} \int_{-\pi}^{\pi}\left|f\circ_i T_a\left(e^{i\theta} \right) - f(a)\right|^2\,d\theta}.\]
This semi-norm does however not satisfy any invariance property under M\"obius transformations.
\end{remark}

\begin{lemma}\label{BMOCarleson}
If $f\in\lhol(\D)$, then the following statements are equivalent.
\begin{enumerate}[(i)]
\item \label{azz}The function $f$ belongs to $\BMOSH(\D)$.
\item \label{azy}The measure $\mu_f$ given by
\begin{equation}\label{MuF}
 d\mu_f(s) = \frac1{s_1^2}(1-|s|^2)\left|\sderiv f(s)\right|^2 \,d\lambda_4(s)
 \end{equation}
is a Carleson measure.
\item\label{azx} The measure $\nu_f$ given by
\[ d\nu_f(s) = \frac{1}{s_1^2}\log\left(\frac{1}{|s|}\right)\left|\sderiv f(s)\right|^2 \,d\lambda_4(s)\]
is a Carleson measure.
\end{enumerate}
\end{lemma}
\begin{proof}
We prove the equivalence  (\ref{azz}) and (\ref{azy}). The equivalence of (\ref{azz}) and (\ref{azx}) follows by analogous arguments. Moreover, we assume that $f$ is not constant, in which case the above statements are obviously all true.

Let $i,j \in\S$ and write $f_i = f_1 + f_2j$ with holomorphic components according to \Cref{SplitLem}. By \Cref{BMOSplit}, the function $f$ belongs to $\BMOSH(\D)$ if and only if $f_1,f_2\in\BMOA(\D_i)$. By \cite[Theorem~6.5]{Girela}, this is equivalent to 
\[ d\mu_{f_\ell}(z) = (1-|z|^2)\left|f_\ell'(z)\right|^2\, d\lambda_2(z), \quad \ell \in\{1,2\}\]
being Carleson measures on $\D_i$. Now observe that $\sderiv f(z) = f_1'(z) + f_2'(z)j$ and hence $|\sderiv f(z)|^2 = |f_1'(z)|^2 + |f_2'(z)|^2$ for $z\in\D_i$. Since $\mu_{f_\ell}$ for $\ell\in\{1,2\}$ are positive measures, we find that $f\in\BMOSH(\D)$ if and only if for any $i\in\S$ the measure $\rho_i = \mu_{f_1} + \mu_{f_2}$, which is given by
\[ d\rho_i(z) = (1-|z|^2) \left|\sderiv f(z)\right|^2\,d\lambda_2(z),\]
is a Carleson measure on $\D_i$.

Observe that the \Cref{RepFo} allows us to estimate the measure $\rho_{\jmath}$ for any $\jmath\in\S$ by the measure $\rho_i$: for any $E\in\borel(\C_{\jmath})$, we can set 
\[E_i = \{z_0+iz_1: z_0+\jmath z_1\in E\}\]
 and obtain
\begin{gather*}
\rho_{\jmath}(E) = \int_{E}(1-|z|^2)\left|\sderiv f(z)\right|^2\,dz\\
= \int_{E}(1-|z|^2)\left|\frac{1}{2}(1-\jmath i)\sderiv f(z_i) + \frac{1}{2}(1+\jmath i)\sderiv f(\overline{z_i})\right|^2\,dz\\
\leq 2 \int_{E_i}(1-|z|^2)\left|\sderiv f(z)\right|^2\,dz +  2 \int_{\overline{E_i}}(1-|z|^2)\left|\sderiv f(z)\right|^2\,dz
\end{gather*}
such that
\begin{equation}\label{RhoIEst}
 \rho_{\jmath}(E) \leq 2 \rho_i(E_i) + 2 \rho_i(\overline{E_i})
\end{equation}

We set $C(i) := \rho_i(\D_i^+)$. \Cref{RepFo} implies that the map $i\mapsto C(i)$ is continuous and hence the constants  $C(i)$ are bounded uniformly from below by a constant $\tau>0$. Otherwise there exists a sequence $(i_n)_{n\in\N}$ of imaginary units such that $C(i_n)\to 0$ as $n\to \infty$. If necessary after passing to a subsequence, we may assume that the $i_n$ converge to some $i_0\in\S$ and we obtain $C(i_0) = \lim_{n\to+\infty}C(i_n) = 0$. This however implies that $\sderiv f \equiv 0$ on $\D_{i_0}$ and so in turn  \Cref{IDPrinciple} implies that $\sderiv f \equiv 0$ on all of $\D$. Hence, $f$ is constant, but this case was excluded by the assumption made above. 

We define now measures $\mu_{f,i}^+$ for $i\in\S$ by setting $\mu_{f,i}^+(E):=\frac{1}{C(i)}\rho_i(E)$ for all $E\in\borel(\D_i^+)$ and a measure $\nu$ on $\S$ by $d\nu(i) := C(i)\,d\sigma(i)$, where $\sigma$ is the surface measure on $\S$. By \Cref{ExFMeas}, these measures constitute a slice decomposition of $\mu_f$.

If now $f\in\BMOSH(\D)$, then the measures $\mu_{f,i}^{+}$ are uniformly Carleson: indeed, if $N$ is the Carleson bound of the measure $\rho_i$ for some fixed $i\in\S$, then \eqref{RhoIEst} implies for any $\theta_0\in[0,2\pi)$, any $h>0$, and any $\jmath\in\S$ that
\begin{gather*}
\mu_{f,\jmath}^+\left(S_{\jmath}(\theta_0,h)\right)  = \frac{1}{C(\jmath)} \rho_{j}\left(S_{\jmath}(\theta_0,h)\cap\C_{\jmath}^+\right)\\
 \leq \frac{2}{\tau}\left(\rho_i\left(S_{i}(\theta_0,h)\right) + \rho_i\left(S_{i}(2\pi-\theta_0,h)\right) \right) \leq \frac{4N}{\tau}h.
\end{gather*}
Hence, the slice decomposition of $\mu_f$ is slice Carleson and we deduce from \Cref{SC->C} that $\mu_f$ is a Carleson measure.

If on the other hand $f\notin\BMOSH(\D)$ and we choose an arbitrary $i\in\S$, then  $\rho_i$ is not Carleson. Hence, we can find two sequences $(\theta_n)_{n\in\N}$ and $(h_n)_{n\in\N}$ such that after setting $S_{i,n}:= S_i(\theta_n,h_n) + \overline{S_i(\theta_n,h_n)}$ we have
\[4nh_n < \rho_i(S_i(\theta_n,h_n)) \leq \rho_i(S_{i,n}).\]
If we set $S_{\jmath,n} = \{s_0 + \jmath s_1: s_0 + is_1\in S_{i,n}\}$ for arbitrary $\jmath\in\S$, then we deduce from \eqref{RhoIEst} because of $\overline{S_{\jmath,n} }= S_{\jmath,n}$ that
\[ 4nh_n < \rho_i(S_{i,n}) \leq 4 \rho_{\jmath}(S_{\jmath,n})\]
and hence $nh_n \leq \rho(S_{\jmath,n})$. Now observe that $S_{\jmath,n} = S(\theta_n,h_n)\cap\C_{\jmath}$ such that
\begin{gather*}
\mu_f(S(\theta_n,h_n)) = \int_{\S}\mu_{f,i}^{+}(S(\theta_n,h_n)\cap \D_i^+) C(i)\,d\sigma(i)\\
 = \int_{\S}\rho_i(S(\theta_n,h_n)\cap \D_i^+)\,d\sigma(i)\\
 = \frac{1}{2}  \int_{\S}\rho_i(S_{i,n})\,d\sigma(i) > \frac{\sigma(\S)}{2}nh_n,
\end{gather*}
where the third equality follows because $\sigma$ is invariant under the mapping $i\mapsto -i$ (cf. \Cref{ArcozziRem}). 
As a consequence $\mu_f$ is not a Carleson measure and the proof is complete.

\end{proof}
\begin{remark}
One could also wonder if a function $f\in\lhol(\D)$ belongs---just in analogy with the complex case---to $\BMOSH(\D)$ if and only if the measures defined by 
\[d\tilde{\mu}_f(s) = (1-|s|)^2\left|\sderiv f(s)\right|^2\,d\lambda_4(s)\]
resp. 
\[d\tilde{\nu}_f(s)  = \log(1/|s|^2)\left|\sderiv f(s)\right|^2\,d\lambda_4(s)\]
are Carleson measures. This is however not true. An easy counterexample is the function $f(s) = 1/\sqrt{1+s}$, where the square root is defined and slice hyperholomorphic on $\H\setminus(-\infty,0]$.

Indeed, for $i\in\S$, we have
\[ \int_{-\pi}^{\pi} |f(e^{i\theta})|^2\,d\theta = \lim_{\varepsilon\to 0^+}\int_{-(\pi-\varepsilon)}^{\pi - \varepsilon} \frac{1}{\sqrt{2+2\cos\theta}}\,d\theta.\]
A primitive of the function $1/\sqrt{2+2\cos\theta}$ on the interval $(-\pi,\pi)$ is given by $F(\theta) = -\log\left(\frac{1-\tan(\theta/4) }{1+\tan(\theta/4)}\right)$ such that
\begin{align*}
 &\int_{-\pi}^{\pi} |f(e^{i\theta})|\,d\theta \\
 =&\lim_{\varepsilon\to 0^+}\left(-\log\left( \frac{1-\tan\left(\frac{\pi - \varepsilon}{4}\right)}{1+\tan\left(\frac{\pi - \varepsilon}{4}\right)}\right)+ \log\left(\frac{1-\tan\left(\frac{-\pi + \varepsilon}{4}\right)}{1+\tan\left(\frac{-\pi + \varepsilon}{4}\right)}\right)\right) = +\infty.
\end{align*}
Hence, the restriction of $f$ to $\D_i$, which coincides with the holomorphic component function $f_1$ obtained from \Cref{SplitLem}, does not belong to the Hardy space $\hardy_{\C}^2(\D_i)$ and in turn neither to $\BMOA(\C_i)$. By \Cref{BMOSplit}, we thus have $f\notin\BMOSH(\D)$. Nevertheless, from \Cref{ExFMeas}, we obtain that a slice decomposition of $\tilde{\mu}_f$ is given by
\begin{equation*}
\left\{\begin{array}{rcl}
\displaystyle d\nu(i) &=& \displaystyle C \,d\sigma(i)\\
\displaystyle d\tilde{\mu}_{f,i}^+(z)& =&\displaystyle \frac{1}{C}  (1-|z|^2)z_1^2 |\sderiv f(z)|^2\,d\lambda_2(z),
\end{array}\right.
\end{equation*}
where $z= z_0 + i z_1$ and with $C = C(i) = \int_{\D_i^+}(1-|z|^2)z_1^2 |\sderiv f(z)|^2\,d\lambda_2(z)$ and $\partial_Sf(s) = -\frac12(1+s)^{-\frac{3}{2}}$.  Observe that the densities of the measures $\tilde{\mu}_{f,i}^+$ are uniformly bounded, because the term $(1-|z|^2)z_1^2$ compensates the singularity at $-1$: for $z=-1+re^{i\theta}\in\D_i$, we have
\begin{gather*}
(1-|z|^2)z_1^2 |\sderiv f(z)|^2 = \left(1 - (1- r\cos(\theta))^2 -r^2\sin(\theta)^2\right) r^2\sin(\theta)^2\frac{1}{4}\frac{1}{r^3}\\
 = \frac{1}{4}(2\cos(\theta) - r)\sin(\theta)^2 < K<+\infty.
\end{gather*}
Thus,  $C<+\infty$ and the measures $\mu_i^+$ are uniformly Carleson. From \Cref{SC->C} we thus deduce that $\tilde{\mu}_{f}$ is a Carleson measure on $\D$ although $f\notin\BMOSH(\D)$. Similar computations can be done for $\tilde{\nu}_f$. 

\end{remark}

\section{The space $\VMOSH(\D)$}\label{VMOSec}

The space $\VMOA(\D_{\C})$ of holomorphic functions of vanishing mean oscillation on the complex unit disc $\D_\C$ is the space of all functions $f\in\BMOA(\D_{\C})$ such that
\[ \lim_{I\subset\R,|I|\to 0} \int_{I}\left| f\left(e^{\imath\theta}\right)- f_I\right|^2\,d\theta = 0,\]
where we use the same notation as in the beginning of \Cref{BMOSect}. This motivates the following definition.
\begin{definition}
The space $\VMOSH(\D_i)$ consists of all those functions $f\in\BMOSH(\D)$ such that 
\begin{equation}\label{VMOSRCond}
\lim_{I\subset\R, |I|\to 0} \frac{1}{|I|} \int_{I} \left|f\left(e^{i\theta}\right) - f_{I,i}\right|^2\,d\theta = 0\qquad\forall i\in\S.
\end{equation}
\end{definition}
\begin{remark}\label{VMOComment}
Observe that it is sufficient that \eqref{VMOSRCond} holds true for one $i\in\S$. The representation formula then implies that it holds true for all $i\in\S$, cf. the proof of \cite[Proposition~5.2]{Arcozzi}.
\end{remark}

\begin{lemma}\label{VMOSplit}
Let $i,j\in\S$ with $i\perp j$, let $f\in\lhol(\D)$ and write $f_i = f_1 +f_2j$ according to \Cref{SplitLem} with holomorphic components $f_1,f_2:\D_i\to\C_i$. Then $f\in\VMOSH(\D)$ if and only if its components $f_1$ and $f_2$ belong to $\VMOA(\D_i)$.
\end{lemma}
\begin{proof}
Recall formula \eqref{AGH}, which establishes a relation between the integral in \eqref{VMOSRCond} and the respective integrals for the component functions:
\begin{equation*}\begin{split}
& \frac{1}{|I|} \int_{I} |f(e^{i\theta}) - f_{I,i}|^2\,d\theta \\
=&\frac{1}{|I|} \int_{I} |f_1(e^{i\theta}) - f_{I,1}|^2\,d\theta +  \frac{1}{|I|} \int_{I} |f_2(e^{i\theta}) - f_{I,2}|^2\,d\theta.
\end{split}\end{equation*}
\Cref{VMOComment} implies now that $f$ belongs $\VMOSH(\D)$ if and only if the left-hand side tends to zero as $|I|$ tends to zero. But this happens obviously if and only if both integrals on the right-hand side tend to zero as $|I|$ tends to zero, i.e. if and only if $f_1$ and $f_2$ belong to the classical $\VMOA(\D_i)$. .

\end{proof}

\begin{lemma}\label{VMOC}
Let $f\in\BMOSH(\D)$. Then the following statements are equivalent.
\begin{enumerate}[(i)]
\item\label{VMOCi} $f\in\VMOSH(\D)$.
\item\label{VMOCii} The functions $f_r$, which are defined for $r\in[0,1]$ as $f_r(s) := f(rs)$, tend to $f$ in $\BMOSH(\D)$ as $r\to 1^+$. 
\item\label{VMOCiii} $f$ belongs to the $\BMOSH$-closure of 
\[\mathcal{A}(\D) := \left\{f\in\lhol(\D): \text{$f$ has a continuous extension to $\overline{\D}$}\right\}.\] 
\item\label{VMOCiv} $f$ belongs to the $\BMOSH$-closure of the set of left slice hyperholomorphic polynomials. 
\end{enumerate}
Consequently, $\VMOSH(\D)$ is a closed subspace of $\BMOSH(\D)$. 
\end{lemma}

\begin{proof}
Let $i,j\in\S$ with $i\perp j$ and apply \Cref{SplitLem} in order to write $f_i = f_1 + f_2j$ with holomorphic functions $f_1,f_2: \D_i\to\C_i$. By \Cref{VMOSplit}, $f\in\VMOSH(\D_i)$ if and only if $f_1,f_2\in\VMOA(\D_i)$. Moreover, we saw in \eqref{BMOSplitEst} that
\begin{equation}\label{KKK}
 \| f_{\ell} \|_{\BMOA(\D_i)} \lesssim  \|f\|_{\BMOSH(\D)} \lesssim \|f_1\|_{\BMOA(\D_i)} +\|f_2\|_{\BMOA(\D_i)}.
 \end{equation}
Thus, $f_r$ tends to $f$ in $\BMOSH(\D)$ if and only if $f_{1,r}(z):=f_1(rz)$ and $f_{2,r}(z):=f_{2}(rz)$ tend to $f_1$ resp. $f_2$ in $\BMOA(\D_i)$. By \cite[Theorem~2.1, Theorem~5.5]{Girela} this is equivalent to  $f_1,f_2\in\VMOA(\D_i)$, which is  by \Cref{VMOSplit} in turn equivalent to $f\in\VMOSH(\D)$. Hence, $(\ref{VMOCi})\Leftrightarrow(\ref{VMOCii})$.

Similarly, we show the equivalence of $(\ref{VMOCi})$ and $(\ref{VMOCiii})$: the function $f$ belongs to $\VMOSH(\D)$ if and only if $f_1,f_2\in\VMOA(\D_i)$. By \cite[Theorem~5.5]{Girela}, this is equivalent to the existence of functions $\tilde{f}_{n,1},\tilde{f}_{n,2}\in\mathcal{A}_{\C}(\D_i)$ such that $\tilde{f}_{n,\ell} \to f_{\ell}$ in $\BMOA(\D_i)$. By \Cref{SplitLem} and \Cref{ExtLem} and \eqref{KKK} this is in turn equivalent to the existence of functions $f_n\in\mathcal{A}(\D)$ such that $f_n\to f$  in $\BMOSH(\D)$.

Finally, the equivalence of $(\ref{VMOCi})$ and $(\ref{VMOCiv})$ follows again by analogous arguments from the respective result in the complex case in \cite[Theorem~5.5]{Girela} and  \Cref{PolyExt}.

\end{proof}

\begin{remark}
In the classical complex theory, one can choose in item (\ref{VMOCii}) of the preceeding lemma $r\in\C$ with $|r|\leq 1$ but not necessarily real. In the quaternionic case this is not possible because multiplying the argument with a constant factor does not preserve slice hyperholomorphicity unless this factor is real.
\end{remark}

\begin{lemma}
Let $f\in\lhol(\D)$ and let $\mu_f$ and $\nu_f$ be defined as in \Cref{BMOCarleson}. The following statements are equivalent:
\begin{enumerate}[(i)]
\item The function $f$ belongs to $\VMOSH(\D)$.
\item The measure $\mu_f$ is a vanishing Carleson measure.
\item The measure $\nu_f$ is a vanishing Carleson measure.
\end{enumerate}
\end{lemma}
\begin{proof}
We adopt the notation used in the proof of \Cref{BMOCarleson}. From the complex theory, cf. \cite[Theorem~6.6]{Girela}, and \Cref{VMOSplit}, it follows that $f\in\VMOSH(\D)$ if and only if the measures $\mu_{f_1}$ and $\mu_{f_2}$ and in turn also the measure $\rho_i$ are vanishing Carleson measures for any choice of $i,j\in\S$ with $i\perp j$.

If $f\in\VMOSH(\D)$, choose $i\in\S$ and let $b(\varepsilon)$ with $\lim_{\varepsilon\to 0^+}b(\varepsilon) \to 0$ be such that $\rho_i(S_i(\theta_0,h))\leq h b(h)$. Then \eqref{RhoIEst} implies for any $\jmath\in\S$, any $\theta_0\in[0,\pi]$, and any $h>0$ that
\begin{gather*}
 \mu_{f,\jmath}^+(S_{\jmath}(\theta_0,h)) \leq \rho_{\jmath}(S_{\jmath}(\theta_0,h))\\
  \leq 2 \rho_i(S_{i}(\theta_0,h)) + 2 \rho_{i}(S_{i}(2\pi - \theta_0,h)) \leq 4 h b(h).
 \end{gather*}
Thus, the measures $\nu$ and $\mu_{f,i}^+$, which are defined as in the proof of \Cref{BMOCarleson},  constitute a vanishing slice  Carleson decomposition of $\mu_f$. We deduce from \Cref{VSC->VC} that $\mu_f$ is a vanishing Carleson measure.

If on the other hand $f\notin\VMOSH(\D)$, then $\rho_i$ is not a Carleson measure and hence there exist some $\varepsilon >0 $ and sequences $(\theta_n)_{n\in\N}$ and $(h_n)_{n\in\N}$ with $h_n \to 0 $ as $n\to+\infty$ such that with the position $S_{i,n} = S_{i}(\theta_n,h_n) \cup \overline{S_{i}(\theta_n,h_n)}$ we have
\[ 4\varepsilon  < \frac{\rho_i(S_{i}(\theta_n, h_n))}{h_n} \leq \frac{\rho_i(S_{i,n})}{h_n}.\]
If we set $S_{\jmath,n} = \{s_0 + js_1: s_0 + i s_1\in S_{i,n}\}$ for $j\in\S$, then the estimate \eqref{RhoIEst} implies because of $S_{\jmath,n} = \overline{S_{\jmath,n}}$ that
\[ 4\varepsilon < \frac{\rho_i(S_{i,n})}{h_n} \leq 4 \frac{\rho_{\jmath}(S_{\jmath,n})}{h_n}\]
and hence $\varepsilon < \rho_{\jmath}(S_{\jmath,n})/h_n$ for any $\jmath\in\S$. Thus, since $S_{\jmath,n} = S(\theta_n,h_n)\cap\C_{\jmath}$, we have
\begin{gather*}
\frac{\mu_f(S(\theta_n,h_n))}{h_n} = \frac{1}{h_n}\int_{\S} \mu_{f,i}^+(S(\theta_n,h_n)\cap\D_i^+) C(i)\,d\sigma(i)\\
= \frac{1}{h_n}\int_{\S} \rho_i(S(\theta_n,h_n)\cap\D_i^+)\,d\sigma(i)\\
= \frac{1}{2 h_n} \int_{\S} \rho_{i}(S_{i,n})\,d\sigma(i) > \frac{\sigma(\S)}{2}\varepsilon,
\end{gather*}
where the third equality follows because $\sigma$ is invariant under the mapping $i\mapsto -i$. Consequently, $\mu_f$ is not a Carleson measure and hence (i) and (ii) are equivalent.

The equivalence of (i) and (iii) can be shown with similar arguments.

\end{proof}

We conclude this section with two results, one showing that $f \in\VMOSH(\D)$ if there exists a suitable majorant of its slice derivative and  one concerning power series with Hadamard gaps in $\VMOSH(\D)$.

\begin{lemma}Let $f\in\lhol(\D)$. If there exists a  monotone increasing function $\phi(r)$ for $0<r <1 $ such that 
\begin{enumerate}[(i)]
\item $|\sderiv f(s)|\leq \phi(r)$ for all $s\in \D$ with  $|s|=r  $  and any $r\in (0,1) $ and 
\item $\displaystyle   \int_0^1 (1-r^2) \phi^2(r)dr< + \infty ,$
\end{enumerate}
then $f \in \VMOSH(\D)$.
\end{lemma}
\begin{proof}  Let $i,j\in \S$ with $j\perp i$ and apply \Cref{SplitLem} in order to write  $f_i = f_1+f_2 j$ with holomorphic components $f_1$ and $f_2$. As 
\[ |\sderiv f(z)|^2= |f_1'(z)|^2+ |f_2'(z)|^2, \forall z\in \D_i ,\]
we have $ | f_\ell'(z) |\leq \phi(r)$ for   $|z|=r  $  and $\ell=1,2$. Since $(ii)$ holds, we can apply the respective complex result, cf. \cite[p.~25]{danikas}, to conclude that $f_1,f_2\in \VMOA( \D_i )$. This implies by \Cref{VMOSplit} that  $f\in \VMOSH(\D_i)$ and thus we have finished the proof.

 \end{proof}

\begin{lemma}
Let $f(s) = \sum_{\ell=1}^{+\infty} s^{n_{\ell}}a_{\ell}$ with $a_{\ell}\in\H$ and $n_{\ell}\in\N$ such that $n_{\ell+1}/n_{\ell} \geq \alpha >1 $ for all $\ell\in\N$. Then the following statements are equivalent:
\begin{enumerate}[(i)]
\item $f\in\BMOSH(\D)$
\item $f\in\VMOSH(\D)$
\item $f\in\hardy^2(\D)$. 
\end{enumerate} 
\end{lemma}
\begin{proof}
Choose $i,j\in\S$ with $i\perp j$ and write $a_\ell = a_{\ell,1} + a_{\ell,2}j$ with $a_{\ell,1},a_{\ell,2}\in\C_i$ for $\ell\in\N$. Applying \Cref{SplitLem} and writing $f_i = f_1 + f_2j$ with holomorphic components, we find that $f_{1}(z) = \sum_{\ell=1}^{+\infty} a_{\ell,1}z^{n_\ell}$ and $f_{2}(z) = \sum_{\ell=1}^{+\infty} a_{\ell,2}z^{n_\ell}$ are complex power series with Hadamard gaps. The lemma follows therefore from the respective complex result, Theorem~9 in \cite{danikas}, and \Cref{BMOSplit,VMOSplit,HardySplit}.

\end{proof}

\section{Relations to other spaces of slice hyperholomorphic functions}\label{RelSec}
In the complex theory there exist relations between the spaces of holomorphic functions with bounded and vanishing mean oscillation and several other spaces of holomorphic functions. The aim of this last section is to show some of the respective counterparts in the quaternionic setting. 
\begin{definition}
The Bloch space $\bloch(\D)$ of slice hyperholomorphic functions is the space of all functions $f\in\lhol(\D)$ such that 
\begin{equation}\label{BlochNorm}
 \|f\|_{\bloch} = |f(0)| + \sup_{s\in\D}\, (1-|s|)^2\left|\sderiv f(s)\right| < +\infty.
\end{equation}
The little Bloch space $\bloch_0(\D)$ of slice hyperholomorphic functions is the space of all functions $f\in\bloch(\D)$ such that 
\begin{equation}
\lim_{|s|\nearrow 1} (1-|s|^2)\left|\sderiv f(s)\right| = 0.
\end{equation}
\end{definition}
\begin{remark}
The Bloch space is a non-separable quaternionic Banach space with the norm defined in \eqref{BlochNorm} and the little Bloch space is the closure of the set of slice hyperholomorphic polynomials in $\bloch(\D)$ and hence a separable subspace of $\bloch(\D)$. These spaces were studied in \cite{CGS3}. Moreover, as it usually happens in theory of slice hyperholomorphic functions, if one takes the supremum in \eqref{BlochNorm} only over one slice $\D_i$ instead of the entire ball $\D$, then one obtains the norm $\|f \|_{\bloch(\D_i)}$, which is equivalent to $\|f\|_{\bloch(\D)}$ as a consequence of the representation formula. Furthermore, if $i,j\in\S$ with $i\perp j$ and we write $f_i = f_1 + f_2j$ by applying \Cref{SplitLem}, then $f$ belongs to $\bloch(\D)$ resp. $\bloch_0(\D)$ if and only if the component functions $f_1,f_2$ belong to the respective complex spaces $\bloch_{\C}(\D_i)$ and $\bloch_{\C,0}(\D_i)$. 
\end{remark}
\begin{lemma}
For $f\in\BMOSH(\D)$, we have
\begin{equation}\label{BlochEstim}
 \|f\|_{\bloch} \lesssim \|f\|_{\BMOSH(\D)}.
 \end{equation}
Hence $\BMOSH(\D)\subset\bloch(\D)$ and $\VMOSH(\D) \subset \bloch_0(\D)$, where these inclusions are continuous.  
\end{lemma}
\begin{proof}
Choose $i,j\in\S$ with $i\perp j$ and write $f_i = f_1 + f_2j$ with holomorphic components according to Lemma~\ref{SplitLem}. By Corollary~5.2 and Theorem~5.1 in \cite{Girela}, there exists a constant $C$ such that 
\[ \| f_{\ell} \|_{\bloch_{\C}(\D_i)} \leq C \|f_{\ell}\|_{\BMOA(\D_i)}\quad\text{for }\ell\in\{1,2\}.\]
We thus have by \Cref{BMOSplit} and because of $|f_{\ell}(z)|\leq |f_i(z)|$ for $\ell\in\{1,2\}$ that
\begin{gather*}
 \|f\|_{\bloch(\D)}\lesssim \|f\|_{\bloch(\D_i)} \leq \|f_1\|_{\bloch_{\C}(\D_i)} + \|f_2\|_{\bloch_{\C}(\D_i)}\\
 \leq C \| f_1\|_{\BMOA(\D_i)} + C\|f_2\|_{\BMOA(\D_i)} \leq 2C \| f\|_{\BMOSH(\D_i)} \lesssim \| f\|_{\BMOSH(\D)}
 \end{gather*}
Consequently $\BMOSH(\D)$ is continuously embedded into $\bloch(\D)$. Furthermore any function in $\VMOSH(\D)$ can be approximated by  slice hyperholomorphic polynomials in $\BMOSH(\D)$ by \Cref{VMOC} and thus, because of \eqref{BlochEstim}, also in $\bloch(\D)$. Since the little Bloch space is the closure of such polynomials in $\bloch(\D)$, we find $\VMOSH(\D)\subset\bloch_0(\D)$.

\end{proof}

\begin{definition}
The Dirichlet space $\dirichlet(\D)$ of slice hyperholomorphic functions is the space of all functions $f\in\lhol(\D)$ such that
\[\sup_{i\in\S}\int_{\D_i}\left|\sderiv f(s)\right|^2\,\lambda_2(s) < +\infty \] 
\end{definition}
\begin{remark}
The Dirichlet space $\dirichlet(\D)$ is a quaternionic right Hilbert space, when it is equipped with the scalar product
\[ \langle f,g\rangle := \overline{ f(0)}g(0) + \int_{\D_i} \overline{\sderiv f(s)}\sderiv g(z)\,d\lambda_2(s),\]
where $i$ is some imaginary unit in $\S$. This scalar product is independent of the choice of $i$: if we write $f$ and $g$ as Taylor series $f(z) = \sum_{n=0}^{+\infty}s^na_n$ and $g(s) = \sum_{n=0}^{+\infty}s^nb_n$, then
\begin{align*}
 \langle f,g\rangle =& \overline{f(0)}g(0)+ \sum_{n=0}^{+\infty} \sum_{k=0}^{+\infty} \overline{(n+1)a_{n+1} }\int_{\D_i} \overline{s}^{n} s^{k}\,d\lambda_2(z) \, (k+1)b_{k+1}
\\
=& \overline{f(0)}g(0)+ \sum_{n=0}^{+\infty} \pi(n+1)\overline{a_{n+1}}b_{n+1},
\end{align*}
 because $\int_{\D_i} \overline{s}^{k+1} s^{n+1}\,d\lambda_2(z) = 0$ if $k \neq n$ and $\int_{\D_i} \overline{s}^{n+1} s^{n+1}\,d\lambda_2(z) = \pi/(n+1)$. The Dirichlet space was also studied in \cite{CGS3} and, as usually, if we choose $i,j\in\S$ with $i\perp j$ and write $f_i = f_1 + f_2j$ with holomorphic components according to \Cref{SplitLem}, then the function $f$ belongs to $\dirichlet(\D)$ if and only if $f_1$ and $f_2$ belong to the complex Dirichlet space $\dirichlet_{\C}(\D_i)$. 

\end{remark}
\begin{lemma}
We have $\dirichlet(\D)\subset\VMOSH(\D)$.
\end{lemma}
\begin{proof}
Let $f\in\lhol(\D)$, choose $i,j\in\S$ and write $f_i = f_1 + f_2 j$ with holomorphic components according to \Cref{SplitLem}. If $f\in\dirichlet(\D)$, then $f_1$ and $f_2$ belong to the complex Dirichlet space $\dirichlet_{\D}(\D_i)$. By \cite[Theorem~10]{danikas}, this implies that $f_1,f_2\in\VMOA(\D_i)$ and hence we conclude from \Cref{VMOSplit} that $f\in\VMOSH(\D)$.

\end{proof}

We conclude the paper with a version of Fefferman's duality theorem. Note that the identification $(\hardy^1(\D))^*\cong\BMOSH(\D)$ has already been shown in \cite{Arcozzi}.

\begin{theorem}
We have $(\VMOSH(\D))^*\cong \hardy^1(\D)$ and $(\hardy^1(\D))^* \cong\BMOSH(\D)$ under the integral pairing
\begin{equation}\label{pairing}
\langle f,g\rangle = \frac{1}{2\pi}\int_{\partial(\D_i)}\overline{f\left(e^{i\theta}\right)}\,g\left(e^{i\theta}\right)\,d\theta,
\end{equation}
where $i\in\S$ is an arbitrary imaginary unit. The operator norm and the ${\|\cdot\|_{\hardy^1(\D)}}$-norm are equivalent norms on $(\VMOSH(\D))^*$. Similarly, the operator norm and the $\|\cdot\|_{\BMOSH(\D)}$-norm are equivalent norms on $(\hardy^1(\D))^*$. Moreover, the identification is independent of the choice of the imaginary unit $i\in\S$.
\end{theorem}
\begin{proof}
Let first $f\in\hardy^1(\D)$. We choose $i,j\in\S$ with $i\perp j$ and write $f_i=f_1 + f_2j$ according to \Cref{SplitLem}. By  \Cref{VMOSplit}, we have $f_1,f_2\in\hardy^1(\D_i)$ and we deduce from Theorem~7.1 in \cite{Girela} that there exist linear maps $\varphi_{\ell}\in(\VMOA(\D_i))^*$ for $\ell\in\{1,2\}$ such that 
\begin{equation}\label{PhiF}
\varphi_{\ell}(g) = \frac{1}{2\pi}\int_{\partial\D_i}\overline{f_{\ell}\left( e^{i\theta}\right)}\, g\left(e^{i\theta}\right)\,d\theta \qquad\text{for }g\in\VMOA(\D_i).
\end{equation}

Now write $g_i = g_1 + g_2 j$ for every $g\in\VMOSH(\D)$ with $g_1,g_2\in\VMOA(\D_i)$ according to \Cref{VMOSplit} and set
 \begin{align*}
\varphi(g) := & \varphi_1(g_1) + \overline{\varphi_2(g_2)} + j\left( \overline{ \varphi_1(g_2)} - \varphi_2(g_1)\right).
 \end{align*}
 Then $\varphi$ is a quaternionic linear functional on $\VMOSH(\D)$ and we have
 \begin{equation}\label{PhiPhiiEst}
 \|\varphi_1\| + \|\varphi_2\| \lesssim \| \varphi\| \lesssim \|\varphi_1\| + \| \varphi_2\|.
 \end{equation}
 Indeed, it is
 \begin{align*}
 \| \varphi_1\| + \|\varphi_2\| \leq & \sup\left\{ |\varphi_1(g)| : g\in\VMOA(\D_i), \|g\|_{\BMOA(\D_i)} = 1\right\}\\
 & + \sup\left\{ |\varphi_2(g)| : g\in\VMOA(\D_i), \|g\|_{\BMOA(\D_i)} = 1\right\}.
 \end{align*}
 Since $\varphi_{\ell}(g)\in\C_i$, we have $|\varphi_{\ell}(g)| \leq |\varphi_1(g) - j\varphi_2(g)| = |\varphi(\ext(g))|$ for $\ell\in\{1,2\}$ and in turn
  \begin{align*}
 \| \varphi_1\| + \|\varphi_2\| \leq & 2 \sup\left\{ |\varphi(\ext(g))| : g\in\VMOA(\D_i), \|g\|_{\BMOA(\D_i)} = 1\right\}\\
 \leq & 2 \sup\left\{ |\varphi(g)| : g\in\VMOSH(\D), \|g\|_{\BMOSH(\D)} \leq C\right\} = \frac{2}{C}\|\varphi\|, 
 \end{align*}
 where $C$ is the constant in \eqref{BMOSplitEst} such that $\|\ext(g) \|_{\BMOSH(\D_i)} \leq C\| g\|_{\BMOA(\D_i)}$. On the other hand, writing again $g_i = g_1 + g_2j$ according to \Cref{SplitLem}, we have
 \begin{align*}
 \| \varphi \| \leq & \sup\left\{ \sum_{\kappa,\ell = 1}^2|\varphi_{\ell}(g_{\kappa}) | : g\in\VMOSH(\D), \| g\|_{\BMOSH(\D_i)} \leq 1 \right\}\\
 \leq &  2\sup\left\{ |\varphi_{1}(g) | : g\in\VMOA(\D_i), \| g\|_{\BMOA(\D_i)} \leq \widetilde{C} \right\}\\
 & +  2 \sup\left\{ |\varphi_{2}(g) | : g\in\VMOA(\D_i), \| g\|_{\BMOA(\D_i)} \leq \widetilde{C} \right\} \\
 \leq &\frac{2}{\widetilde{C}} (\|\varphi_1\| + \|\varphi_2\|),
 \end{align*}
 where $\widetilde{C}>0$ is the constant in \eqref{BMOSplitEst} such that $\| g_{\ell}\|_{\BMOA(\D_i)} \leq \widetilde{C}\|g\|_{\BMOSH(\D_i)}$ for $g\in\BMOSH(\D_i)$. 
 
 Since \eqref{PhiPhiiEst} holds true, $\varphi$ is even continuous on $\VMOSH(\D)$ and hence $\varphi\in(\VMOSH(\D))^*$. Moreover
\begin{equation}\begin{split} \label{SomeCalc}
  \varphi(g) = & \varphi_1(g_1) +\varphi_1(g_2)j - j \varphi_2(g_1) - j\varphi_2(g_2)j\\
=& \frac{1}{2\pi}\int_{\partial\D_i}\overline{f_{1}\left( e^{i\theta}\right)}\, g_1\left(e^{i\theta}\right)\,d\theta + \frac{1}{2\pi}\int_{\partial\D_i}\overline{f_{1}\left( e^{i\theta}\right)}\, g_2\left(e^{i\theta}\right)\,d\theta j \\
&- j\frac{1}{2\pi}\int_{\partial\D_i}\overline{f_{2}\left( e^{i\theta}\right)}\, g_1\left(e^{i\theta}\right)\,d\theta - j \frac{1}{2\pi}\int_{\partial\D_i}\overline{f_{2}\left( e^{i\theta}\right)}\, g_2\left(e^{i\theta}\right)\,d\theta j\\
=& \frac{1}{2\pi}\int_{\partial\D_i}\overline{f\left( e^{i\theta}\right)}\, g\left(e^{i\theta}\right)\,d\theta.
 \end{split}\end{equation}
From \eqref{PhiPhiiEst} and  \eqref{HardySplitEst} it follows that
\begin{align*}
\| \varphi\| \leq \|\varphi_1\| + \|\varphi_2\| \lesssim \|f_1\|_{\hardy_{\C}^1(\D_i)} + \|f_2\|_{\hardy_{\C}^1(\D_i)} \lesssim \| f\|_{\hardy^1(\D)î}
\end{align*}
and
\begin{align*}
\| f\|_{\hardy^1(\D)} \lesssim \|f_1\|_{\hardy_{\C}^1(\D_i)} + \| f_2 \|_{\hardy^1_{\C}(\D_i)} \lesssim \|\varphi_1\|  + \|\varphi_2\| \lesssim \| \varphi\|
\end{align*}
 such that altogether
 \[ \| f\|_{\hardy^1(\D)} \lesssim \|\varphi\| \lesssim \|f\|_{\hardy^1(\D)}. \]
 
 If on the other hand $\varphi \in (\VMOSH(\D))^*$, then we can choose again $i,j\in\S$ with $i\perp j$ and write $\varphi(g) = \tilde{\varphi}_1(g) - j\tilde{\varphi}_2(g)$ with $\C_i$-valued functions $\tilde{\varphi}_1$ and $\tilde{\varphi}_2$. These functions are even $\C_i$-linear and hence the maps $\varphi_{\ell}:g \mapsto \tilde{\varphi}_\ell(\ext(g))$ for $g\in\VMOA(\D_i)$ with $\ell\in\{1,2\}$ are $\C_i$-linear functionals on $\VMOA(\D_i)$. Since $\|f\|_{\BMOA(\D_i)}$ and $\|\ext(f)\|_{\BMOSH(\D)}$ are equivalent norms on $\VMOSH(\D)$ because of \eqref{BMOSplitEst}, these functionals are even continuous and hence belong to $(\VMOA(\D_i))^*$. From \cite[Theorem~7.1]{Girela} we therefore deduce the existence of two functions $f_1,f_2\in\hardy_{\C}^1(\D_i)$ such that \eqref{PhiF} holds true. Setting $f := \ext(f_1 + f_2j)\in\hardy^1(\D_i)$, we find by a computation as in \eqref{SomeCalc} that \eqref{pairing} holds true. Hence, $(\VMOSH(\D))^* \cong \hardy^1(\D)$. 
 
 The second identity $(H^1(\D))^*\cong\BMOSH(\D)$ can be shown by analogous arguments or as it was done in \cite{Arcozzi}.
 
 Finally, the identifications are independent of the choice of the unit $i\in\S$ in the integral pairing: if we write $f(s)=\sum_{n=0}^{+\infty}s^nf_n$ and $g(s) = \sum_{n=0}^{+\infty}s_ng_n$ as power series, then
 \begin{align*}
\langle f,g\rangle = \sum_{n,m=0}^{+\infty} \overline{a_n}\left(\frac{1}{2\pi}\int_{\partial(\D_i)}\overline{e^{in\theta}}\,e^{im\theta}\,d\theta\right) b_{m} = \sum_{n=0}^{+\infty}\overline{a_n}b_n
\end{align*}

\end{proof}


\begin{thebibliography}{99}

\bibitem{Adler}
S. L. Adler: {\em Quaternionic Quantum Mechanics and Quantum Fields}, International Series of Monographs on Physics Volume 88, Oxford University Press, New York 1995.

\bibitem{AlpayBook}
D. Alpay: {\em The Schur Algorithm, Reproducing Kernel Spaces and System Theory}, American Mathematical Society, Providence 2001. Translated from the 1998 French original by Stephen S. Wilson, Panoramas et Synth\`eses.

\bibitem{STNormal} 
D. Alpay, F. Colombo, D. P. Kimsey: The spectral theorem for quaternionic unbounded normal operators based on the $S$-spectrum, {\em J. Math. Phys.} {\bf 57(2)} (2016), 023503, 27 pp. 

\bibitem{STUnit}
D. Alpay, F. Colombo, D. P. Kimsey, I. Sabadini:  The spectral theorem for unitary operators based on the $S$-spectrum, {\em Milan J. Math.} {\bf 84(1)} (2016), 41--61.

\bibitem{SHSchur3}
D. Alpay, F. Colombo, I. Lewkowicz, I. Sabadini:  Realizations of slice hyperholomorphic generalized contractive and positive functions, {\em Milan J.
Math.} {\bf 83(1)} (2015), 91--144.

\bibitem{SHSchur2}
 D. Alpay, F. Colombo, I. Sabadini:  Pontryagin-de Branges-Rovnyak spaces of slice hyperholomorphic functions, {\em  J. Anal. Math.} {\bf 121} (2013), 87--125. 

\bibitem{SHSchur1}
D. Alpay, F. Colombo, I. Sabadini: Schur functions and their realizations in the slice hyperholomorphic setting, {\em Integral Equations Operator Theory }
{\bf 72(2)} (2012), 253--289.

\bibitem{ACS} D. Alpay, F. Colombo,  I. Sabadini:
 {\em Slice Hyperholomorphic Schur analysis}, Operator Theory: Advances and Applications, Volume 256, Springer Basel 2017.
 
\bibitem{Fock}
D. Alpay, F. Colombo, I. Sabadini, G. Salomon: The Fock space in the slice hyperholomorphic setting, In: {\em Hypercomplex Analysis: New Perspectives and Applications}, Trends in Mathematics, Birkh\"auser Basel 2014, 43--59.

\bibitem{DisIntBook}
L. Ambrosio, N. Fusco, D. Pallara: {\em Functions of bounded variation and free discontinuity problems}, Oxford Mathematical Monographs, The Clarendon Press, Oxford University Press New York 2000.

\bibitem{Arcozzi}
N. Arcozzi, G. Sarfatti: {\em From Hankel operators to Carleson measures in a quaternionic variable}, preprint, \href{http://arxiv.org/abs/1407.8479v1}{\tt arXiv:1407.8479 [math.CV]}.

\bibitem{CGS3}
C. M. P. Castillo Villalba, F. Colombo, J. Gantner: Bloch, Besov and Dirichlet spaces of slice hyperholomorphic functions, {\em Complex Anal. Oper. Theory} {\bf 9(2)} (2015) 479--517.

\bibitem{Bergman1}
F. Colombo, J. O. Gonz\'{a}lez-Cervantes, M. E. Luna-Elizarraras: I. Sabadini, M. Shapiro: On Two Approaches to the Bergman Theory for Slice Regular Functions, In: {\em  Advances in Hypercomplex Analysis}, Springer INdAM Series Volume 1, Springer Milan 2013, 39--54.

\bibitem{Bergman4}
F. Colombo, J. O. Gonz\'{a}lez-Cervantes, I. Sabadini: Further properties of the Bergman spaces of slice regular functions, {\em Adv. Geom.} {\bf 15(4)} (2015), 469--484.

\bibitem{Bergman2}
F. Colombo, J. O. Gonz\'{a}lez-Cervantes, I. Sabadini: On slice biregular functions and isomorphisms of Bergman spaces, {\em Complex Var. Elliptic Equ.} {\bf 58(7--8)} (2013), 1355--1372.

\bibitem{Bergman3}
F. Colombo, J. O. Gonz\'{a}lez-Cervantes, I. Sabadini: The C-property for slice regular functions and applications to the Bergman space, {\em Complex Var. Elliptic Equ.} {\bf 58(10)} (2013), 1355-1372.

\bibitem{CliffCalc}
F. Colombo, I. Sabadini, D. C. Struppa: A new functional calculus for noncommuting operators, {\em J. Funct. Anal.} {\bf 254(8)} (2008), 2255--2274.

 \bibitem{CSS}
 F. Colombo, I. Sabadini, D. C. Struppa:
 {\em Entire Slice Regular Functions}, to apper in SpringerBriefs in Mathematics,
 preprint, 2015. \href{http://arxiv.org/abs/1512.04215}{\tt arXiv:1512.04215 [math.CV]}.
 
\bibitem{SHBook} 
 F. Colombo, I. Sabadini, D.C. Struppa: {\em Noncommutative Functional Calculus,  Theory
and Applications of Slice Hyperholomorphic Functions}, Progress in Mathematics Volume 289,
Birkh\"auser Basel 2011.

\bibitem{CliffHolo}
F. Colombo, I. Sabadini, D. C. Struppa: Slice monogenic functions, {\em Israel J. Math.} {\bf 171} (2009), 385--403.

\bibitem{danikas} 
N. Danikas, Some Banach spaces of analytic functions, In: {\em Function spaces and complex
analysis} (Eds.: R. Aulaskari, I. Laine), Univ.
Joensuu Dept. Math. Rep. Ser. Volume 2, Univ. Joensuu Joensuu1999, pp. 9--35.

\bibitem{Fueter1}
R. Fueter: Analytische Funktionen einer Quaternionenvariablen, {\em Comment. Math. Helv.} {\bf 4(1)}  (1932), 9--20.

\bibitem{Fueter2}
R. Fueter: Die Funktionentheorie der Differentialgleichungen $\Delta u = 0$ and $\Delta\Delta u = 0$ mit vier reellen Variablen, {\em Comment. Meth. Helv.} { \bf 7(1)}  (1934), 307--330.

\bibitem{GentiliBook}
G. Gentili, C. Stoppato, D. C. Stuppa: {\em Regular functions of a quaternionic variable}, Springer Monographs in Mathematics, Springer Berlin-Heidelberg, 2013.

\bibitem{GentiliStruppa}
G. Gentili, D. C. Struppa: A new approach to Cullen-regular functions of a quaternionic variable, {\em C. R. Math. Acad. Sci. Paris} {\bf 342(10)} (2006), 741--744.

\bibitem{GhilPer}
R. Ghiloni, V. Moretti, A. Perotti: Continuous slice functional calculus in quatenrionic {H}ilbert spaces, {\em Rev. Math. Phys.} {\bf 25(4)} (2013), 1350006, 83 pp.

\bibitem{Girela} 
D. Girela, Analytic functions of bounded mean oscillation, In: {\em Complex function spaces (Mekrij\"arvi, 1999)}. Vol. 4. Univ. Joensuu Dept. Math. Rep. Ser. Univ. Joensuu Joensuu 2001, 61--€"170.

\bibitem{sabsar} 
I. Sabadini, A. Saracco:  {\em Carleson measures for Hardy and Bergman spaces in the quaternionic unit ball}, preprint, \href{http://www.arxiv.org/abs/1601.03031v1}{\tt arXiv:1601.03031 [math.CV]}.

\bibitem{HardyRef} 
G. Sarfatti, {\em Elements of Function Theory in the Unit Ball of Quaternions}, PhD Thesis, Universit\`{a} degli Studi di Firenze, 2013.

\bibitem{SRMoeb} 
C. Stoppato, {\em Regular Moebius transformations of the space of quaternions}, {\em Ann. Global Anal. Geom.} {\bf 39(4)} (2011), 387--401.

\bibitem{kehe}   
K. Zhu, {\em Operator theory in function spaces}, Second Edition, Mathematical Surveys and Monographs Volume 138,  American Mathematical Society Providence 2007.

\end{thebibliography}
 \end{document}